\documentclass[10pt]{amsart}
\usepackage{amsmath,amsfonts,latexsym,graphicx,amssymb,url}
\usepackage{hyperref,pdfsync}
\usepackage{amsmath,txfonts,pifont,bbding,pxfonts}
\usepackage{manfnt}
\usepackage[active]{srcltx}
\usepackage{cases}
\usepackage{wasysym,pstricks, enumerate,subfigure}
\usepackage{lscape,color}
\usepackage{soul}
\usepackage[none]{hyphenat}
\usepackage[mathscr]{euscript}
\usepackage{wrapfig, subfigure,graphicx}
\usepackage{enumitem}
\usepackage{tikz}
\usetikzlibrary{arrows.meta,calc,decorations.pathreplacing}

\newcommand{\dist}{\text{dist}} 

\newcommand{\diver}{\text{\rm div}}
\newcommand{\R}{{\mathbb R}} 
\newcommand{\C}{{\mathbb C}}
\newcommand{\Z}{{\mathbb Z}}

\newcommand{\E}{{\mathbf E}}
\newcommand{\D}{{\mathbf D}}
\newcommand{\B}{{\mathbf B}}
\newcommand{\J}{{\mathbf J}}
\renewcommand{\H}{{\mathbf H}}
\newcommand{\G}{{\mathbf G}}

\newcommand{\n}{{\mathbf n}}
\newcommand{\m}{{\mathbf m}}
\renewcommand{\k}{{\mathbf k}} 

\renewcommand{\i}{{\rm i}}

\newcommand{\p}{\varphi}
\renewcommand{\(}{\left(}
\renewcommand{\)}{\right)}
\newtheorem{theorem}{Theorem}[section]
\newtheorem{corollary}[theorem]{Corollary}
\newtheorem{lemma}[theorem]{Lemma}
\newtheorem{proposition}[theorem]{Proposition}
\newtheorem{definition}[theorem]{Definition}
\newtheorem{remark}[theorem]{Remark}

\newcommand{\cristian}[1]{\textcolor{black}{#1}}

\begin{document}

\title[Generalized Snell's law and Maxwell equations]{Generalized Snell's law and Maxwell equations}
\author[C. E. Guti\'errez and A. Sabra]{Cristian E. Guti\'errez and Ahmad Sabra\\\today}
\providecommand{\keywords}[1]{\textbf{\textit{Index terms---}} #1}
\thanks{2020 AMS Math Subject Classification: 78A40, 35Q61, 35D30.}
\thanks{
C.E.G was partially supported by NSF Grant DMS-1600578, and A. S. was partially supported by the University Research Board, Grants 104107 and 104631 from the American University of Beirut. \today}
\address{Department of Mathematics\\Temple University\\Philadelphia, PA 19122}
\email{gutierre@temple.edu}
\address{Department of Mathematics \\ American University of Beirut\\ Beirut, Lebanon}
\email{asabra@aub.edu.lb}

\begin{abstract}
This paper examines the Maxwell system of electrodynamics within the framework of distributions. 
A primary objective is to establish general boundary conditions for fields at interfaces when the charge and current densities are measures localized on the interface. From this analysis, the paper presents a derivation of the generalized Snell’s law,  along with formulas for the amplitudes of the reflected and transmitted waves in terms of the incident amplitude.

\end{abstract}


\maketitle

\tableofcontents

\section{Introduction}

Metasurfaces or metalenses are ultra-thin layers built with nano-materials that can steer light in unconventional ways. 
In beam shaping, the subject of metasurfaces  is a rapidly growing area of research with diverse practical applications. 
Central to this field is the generalized Snell's law of refraction and reflection, which explains how beams propagate across metasurfaces.  
The law was introduced in the influential works \cite{gaburrocapasso:generalizedsnelllaw2011} and \cite{aietacapasso2012} for planar geometries. Its formulation involves a function defined in a small neighborhood of the metasurface, called the phase discontinuity, and is further discussed in \cite{capasso-yu:2014}.  
A rigorous mathematical derivation of the law for non-planar geometries was first obtained using wave fronts in \cite[Sect. 3]{Gutierrez:17} and later by applying the Fermat principle of least action in \cite[Sect. 2]{GUTIERREZ2021102134}.  
These works also demonstrate the existence of phase discontinuities for various geometric configurations and multiple applications.

Let us recall precisely the generalized Snell's law in vector form.  
Given a surface $\Gamma$ separating two media $I$ and $II$ with refractive indices $n_1$ and $n_2$ respectively, and $f$ a phase function defined on $\Gamma$, the generalized Snell's law of refraction states that if a wave in medium $I$ with unit direction vector $\k_i$ strikes $\Gamma$ at a point $P$, then the wave is refracted into medium $II$ with unit direction vector $\k_t$ satisfying
\begin{equation}\label{eq:generalized snell law}
n_1\,\k_i - n_2\,\k_t = \lambda\,\n(P) + \nabla f(P)
\end{equation}
where $\n(P)$ is the unit normal to $\Gamma$ at $P$, and $\lambda \in \R$, see \cite[Equation (2.4)]{GUTIERREZ2021102134}.  
On the other hand, a wave reflected back into medium $I$ has unit direction vector $\k_r$ and satisfies the generalized law of reflection
\begin{equation}\label{eq:generalized snell law reflection}
n_1\,\k_i - n_1\,\k_r = \lambda'\,\n(P) + \nabla f(P).
\end{equation}

A primary goal of this paper is to explore Maxwell's equations from a distributional perspective and derive relationships between the electric and magnetic fields on either side of a boundary or interface—that is, boundary conditions—when the current and charge densities are measures concentrated on the interface.  
While Maxwell's equations are well understood in the classical sense, analyzing them in the setting of generalized functions (or distributions) becomes essential when dealing with discontinuous fields across surfaces; see, for example, \cite{2011idemen:book} and \cite{2017gutierrez:invariance}.  
The main result of this part is Theorem \ref{thm:general boundary conditions for Maxwell equations}, which we believe has independent interest, \cristian{since the discontinuities of the fields are assumed to be measures concentrated on the interface, a novel assumption in this generality.}

Using this analysis, the second objective of the paper is to derive the generalized laws of refraction and reflection directly from Maxwell’s equations, \cristian{which is a novel approach in the metasurface setting.}
To achieve this, we propose representing the transmitted and reflected electric fields as nonlinear waves incorporating the phase discontinuity.  
This representation—detailed in equations \eqref{eq:proposed form of the electric field} and \eqref{eq:proposed form of the electric field back}—enables us to deduce the generalized Snell’s laws from the boundary conditions established in Theorem \ref{thm:general boundary conditions for Maxwell equations}.  
Because the electric fields must satisfy the Maxwell system, this imposes constraints on the phase discontinuity, and the main result of this part is Theorem \ref{thm:MAIN THEOREM GSL}.
Using the boundary conditions obtained, a third objective is to calculate the amplitudes of the transmitted and reflected waves in terms of incident wave, Proposition \ref{prop:formulas for amplitudes}.  \cristian{This is possible under the necessary and sufficient condition \eqref{eq:solvability condition involving A3i} between the amplitude of the incident wave and the wave vectors. This is also a new result in the paper.}

To place the results in broader context, it is worth noting that metasurfaces that refract or reflect beams according to prescribed energy patterns are closely connected to Monge–Ampère type partial differential equations, as discussed in \cite{gutierrez-pallucchini:metasurfacesandMAequations}.  
The analysis of chromatic aberration in metalenses is carried out in \cite{GUTIERREZ2021102134}, with further insights available in \cite{gaburrocapasso:generalizedsnelllaw2011}.  
For applications related to tunable metasurfaces using graphene, see \cite{biswas-gutierrez-low-2018}.  
Recent developments and applications in the field can also be found in \cite{rousseaufelbacq2020}, \cite{Ji:2023aa:advancesinmetasurfacesapplications}, and \cite{Yang:2023aa:moreapplications}.

The paper is organized as follows. In Section \ref{sec:preliminaries}, \cristian{we recall results concerning distributions, outline the assumptions on the fields, and establish formulas needed later.} These are applied in Section \ref{sec:maxwell and boundary conditions} to prove Theorem \ref{thm:general boundary conditions for Maxwell equations}, \cristian{which provides general boundary conditions.}  

The derivation of the generalized Snell law \cristian{from the Maxwell's system} is the subject of Section \ref{sec:generalized snell law proof}, where we employ Theorem \ref{thm:general boundary conditions for Maxwell equations}.  
In Section \ref{sec:boundary conditions for magnetic fields}, we obtain boundary conditions for the magnetic fields, which, combined with the previously derived conditions, are used in Section \ref{sec:calculation of the amplitudes of the reflected and transmitted fields BIS} to derive explicit formulas for the wave amplitudes.
\cristian{Finally, the Appendix \ref{app:proof of exponential lemma} contains a proof of the exponential Lemma \ref{lm:exponential gradient}.}

\section{Preliminaries}\label{sec:preliminaries}

\cristian{In this section, we revisit concepts related to distributions that are essential for analyzing the Maxwell system in this context. Specifically, we derive the formulas that are subsequently utilized to establish general boundary conditions in Theorem \ref{thm:general boundary conditions for Maxwell equations}. These formulas are presented in Propositions \ref{prop:fomulas for curl and divergence} and \ref{thm:representation formula for dg/dt}.
}

Let $\Omega\subset \R^3$ be an open and bounded domain. 
A generalized function or distribution in $\Omega$ is a complex-valued continuous linear functional defined in the class of test functions $\mathcal D(\Omega)=C_0^\infty(\Omega)$ that are infinitely differentiable in $\Omega$ having compact support in $\Omega$ . As usual, $\mathcal D'(\Omega)$ denotes the class of distributions in $\Omega$ \cite{schwartz:theoriedesdistributions}. 
If $g\in \mathcal D'(\Omega)$, then, \cristian{as usual,} $\langle g,\p \rangle$ denotes the value of the distribution $g$ on the test function $\p\in \mathcal D(\Omega)$.

We say that ${\bf G}=(G_1,G_2,G_3)$ is a vector valued distribution in $\Omega$ if each component $G_i\in \mathcal D'(\Omega)$, $1\leq i\leq 3$.
The divergence of ${\bf G}$ is the scalar distribution defined by
\begin{equation}\label{form:divergence}
\langle \nabla \cdot {\bf G}, \varphi\rangle
=-\sum_{i=1}^3 \langle G_i, \varphi_{x_i} \rangle,
\end{equation}
and the curl of ${\bf G}$ is the vector valued distribution in $\Omega$ defined by
\begin{equation}\label{form:curl}
\langle \nabla \times {\bf G},\varphi\rangle=\left(\langle G_2,\varphi_{x_3}\rangle-\langle G_3,\varphi_{x_2}\rangle\right){\bf i}-\left(\langle G_1,\varphi_{x_3}\rangle-\langle G_3,\varphi_{x_1}\rangle\right){\bf j}+\left(\langle G_1,\varphi_{x_2}\rangle-\langle G_2,\varphi_{x_1}\rangle\right){\bf k}.
\end{equation}
Then it follows that
\begin{equation}\label{eq:div of curl zero}
\nabla \cdot (\nabla\times {\bf G})=0,
\end{equation} 
in the sense of distributions.
When the distribution ${\bf G}=(G_1,G_2,G_3)$ is locally integrable in $\Omega$ we obtain from \eqref{form:divergence}, and \eqref{form:curl} that 
\begin{equation*}
\langle \nabla \cdot {\bf G},\varphi\rangle=-\int_{\Omega}{\bf G}\cdot \nabla \varphi\,dx,\qquad
\langle \nabla \times {\bf G},\varphi\rangle=\int_{\Omega} {\bf G}\times \nabla \varphi\,dx.
\end{equation*}

We consider the following configuration. $\Omega$ is a smooth open and bounded domain in $\mathbb R^3$ and $\Gamma$ is a smooth surface that splits $\Omega$ into two disjoint open parts $\Omega_+$ and $\Omega_-$, i.e., $\Omega=\Omega_-\cup \Gamma \cup \Omega_+$, as follows: for every $x_0\in \Gamma$ there exists a ball $B(x_0,r)\subset \Omega$ and $\gamma \in C^1\(B(x_0,r)\)$  such that
\begin{align*}
\Omega_-\cap B(x_0,r)&=\{(x_1,x_2,x_3)\in B(x_0,r): x_3<\gamma(x_1,x_2)\}\\
 \Omega_+\cap B(x_0,r)&=\{(x_1,x_2,x_3)\in B(x_0,r): x_3>\gamma(x_1,x_2)\}.
 \end{align*}
 
 \begin{figure}[h]
\centering
 \begin{tikzpicture}[scale=0.8, line cap=round, line join=round]

\definecolor{pluscol}{RGB}{230,120,40}
\definecolor{minuscol}{RGB}{70,130,200}
\definecolor{gammacol}{RGB}{200,30,30}

\fill[minuscol!30]
(-3.2,1.68)
.. controls (-1.8,0.7) and (0.2,0.55) .. (1.3,0.95)
.. controls (2.2,1.25) and (2.8,1.45) .. (3.2,1.68)
arc[start angle=36.87,end angle=143.13,x radius=4,y radius=2.8];

\fill[pluscol!30]
(-3.2,1.68)
.. controls (-1.8,0.7) and (0.2,0.55) .. (1.3,0.95)
.. controls (2.2,1.25) and (2.8,1.45) .. (3.2,1.68)
arc[start angle=36.87,end angle=-36.87,x radius=4,y radius=2.8]
arc[start angle=-36.87,end angle=-216.87,x radius=4,y radius=2.8];

\draw[thick,dashed] (0,0) ellipse (4 and 2.8);

\draw[very thick, gammacol]
(-3.2,1.68)
.. controls (-1.8,0.7) and (0.2,0.55) .. (1.3,0.95)
.. controls (2.2,1.25) and (2.8,1.45) .. (3.2,1.68);

\node at (3,2.35) {$\Omega$};
\node at (0,-1.95) {$\Omega_-$};
\node at (0,2.0) {$\Omega_+$};
\node at (0.55,1) {{\color{red} $\Gamma$}};
\end{tikzpicture}

\caption{$\Omega=\Omega_-\cup \Gamma\cup \Omega_+.$}

\label{fig:interface_domain}

\end{figure}

We are given fields ${\bf G}_-$ in $\Omega_-$ and ${\bf G}_+$ in $\Omega_+$ satisfying the following
properties
\begin{enumerate}[label=(F\arabic*)]
\item\label{condition 1}  ${\bf G}_{-}\in C^1(\Omega_-), {\bf G}_+ \in C^1(\Omega_+)$.
\item\label{condition 2} The first order derivatives of ${\bf G}_\pm$ are in $L^1(\Omega_\pm)$, respectively.
\item\label{condition 3}  For every $x\in \Gamma$, $\lim_{y\to x, y\in \Omega_-} {\bf G}_{-}(y)$ and $\lim_{y\to x, y\in \Omega_+} {\bf G}_{+}(y)$ exist and are finite.
\end{enumerate}
As a consequence, each $\G_-$ and $\G_+$ can be extended continuously to $\Gamma$ by setting 
\begin{equation*}\G_-(x)=\lim_{y\to x, y\in \Omega_-} {\bf G}_{-}(y);\qquad  
\G_+(x)=\lim_{y\to x, y\in \Omega_+} {\bf G}_{+}(y),
\end{equation*} 
for each $x\in \Gamma$. 
For such fields ${\bf G_-}$ and ${\bf G_+}$, the linear functional $\G$ given by 
\begin{equation}\label{eq:definition of bold G}
\langle \G,\varphi\rangle=\int_{\Omega_-}\G_-(x)\,\varphi(x)\,dx+\int_{\Omega_+}\G_+(x)\,\varphi(x)\,dx
\end{equation}
is a well defined distribution for $\p\in \mathcal D (\Omega)$.
The jump of the fields in $\Gamma$ is defined by
\begin{equation*}\label{eq:jump of fields}
[[{\bf G}(x)]]={\bf G_+}(x)-{\bf G_-}(x),\quad \text{for $x\in \Gamma$}.
\end{equation*}
We then have the following expressions for the curl and divergence of $\G$.
\begin{proposition}\label{prop:fomulas for curl and divergence}
If the field $\G$ satisfies \ref{condition 1}--\ref{condition 3}, then for each $\varphi\in \mathcal D(\Omega)$ we have
\begin{align}\label{eq:formula for divergence}
\langle \nabla \cdot {\bf G},\varphi\rangle
&=\int_{\Gamma} \varphi(x)\, [[{\bf G}(x)]]\cdot {\bf n}\, d\sigma
+\int_{\Omega_-}\varphi\, \nabla \cdot {\bf G_-}\, dx+\int_{\Omega_+}\varphi \nabla \cdot {\bf G_+}\, dx\\ \label{eq:formula for curl}
\langle \nabla \times {\bf G},\varphi\rangle
&=-\int_{\Gamma} \varphi(x)\,  [[{\bf G}(x)]]\times {\bf n}\, d\sigma
+\int_{\Omega_-}\varphi \nabla \times {\bf G_-}\, dx+\int_{\Omega_+}\varphi \nabla \times {\bf G_+}\, dx 
\end{align}
with ${\bf n}$ the unit normal to $\Gamma$ pointing toward $\Omega_+$.
\end{proposition}

\begin{proof}
Given $\varepsilon>0$, define
$$\Omega_-^{\varepsilon}=\{x\in \Omega_-: \dist(x,\Gamma)>\varepsilon\}\qquad\qquad \Omega_+^{\varepsilon}=\{x\in \Omega_+: \dist(x,\Gamma)>\varepsilon\}.$$
For $\varphi\in \mathcal D (\Omega)$, we have from \eqref{form:divergence} and the definition of $\G$ in \eqref{eq:definition of bold G} that 
\begin{align*}
\langle \nabla \cdot{\bf  G}, \varphi\rangle
&=-\int_{\Omega_-} {\bf G_-}\cdot \nabla\varphi\,dx-\int_{\Omega_+}{\bf G_+}\cdot \nabla\varphi\, dx\\
&=-\int_{\Omega_-^{\varepsilon}} {\bf G_-}\cdot \nabla\varphi\,dx
-\int_{ \Omega_-\setminus\Omega_-^{\varepsilon}} {\bf G_-}\cdot \nabla\varphi\,dx-\int_{\Omega_+^{\varepsilon}} {\bf G_+}\cdot \nabla\varphi\,dx
-\int_{ \Omega_+\setminus\Omega_+^{\varepsilon}} {\bf G_+}\cdot \nabla\varphi\,dx\\
&=-(I+II+III+IV).
\end{align*}
$II,IV\to 0$ as $\varepsilon\to 0^+$
because
the extensions $\G_\pm$ are locally bounded in $\Omega_\pm\cup \Gamma$ and $\varphi$ is compactly supported in $\Omega$.
Since $\G_-$ is $C^1$ in $\Omega^{\varepsilon}_-$, using the divergence theorem we obtain
$$
I=\int_{\partial \Omega_-^\varepsilon}\varphi {\bf G_-}\cdot {\bf n_-^{\varepsilon}}\,d\sigma-\int_{\Omega_-^\varepsilon} \varphi{\bf \nabla\cdot G_-}\, dx\\
$$
with ${\bf n}^{\varepsilon}_-$ the outward unit normal to $\partial \Omega^{\varepsilon}_-$.
Since $\varphi$ has compact support in $\Omega$, $\nabla \cdot {\bf G_-}\in L^1(\Omega_-)$, and $\G_-$ is locally bounded in $\Omega_-\cup \Gamma$, it follows that 
$$I\to \int_{\Gamma} \varphi {\bf G_-}\cdot {\bf n}\,d\sigma-\int_{\Omega_-}\varphi\nabla\cdot {\bf G_-}\, dx,$$
with ${\bf n}$ is the unit normal to $\Gamma$ toward $\Omega_+$.
Similarly  we get
$$III\to \int_{\Gamma} \varphi {\bf G_+}\cdot ({-\bf n})\,d\sigma-\int_{\Omega_+}\varphi\nabla\cdot {\bf G_+}\, dx,$$
with $-{\bf n}$ the unit normal to $\Gamma$ toward $\Omega_-$.
Hence \eqref{eq:formula for divergence} follows.

We next prove \eqref{eq:formula for curl}.
Let $\varphi\in \mathcal D(\Omega)$, we have from \eqref{form:curl} and the definition of ${\bf G}$ in \eqref{eq:definition of bold G} that
\begin{align*}
\langle \nabla \times {\bf G}, \varphi\rangle&=\int_{\Omega_-} {\bf G_-} \times \nabla \varphi\, dx+\int_{\Omega_+} {\bf G_+} \times \nabla \varphi\, dx\\
&=\int_{\Omega_-^{\varepsilon}} {\bf G_-}\times \nabla\varphi\,dx
+\int_{ \Omega_-\setminus\Omega_-^{\varepsilon}} {\bf G_-}\times \nabla\varphi\,dx+\int_{\Omega_+^{\varepsilon}} {\bf G_+}\times \nabla\varphi\,dx+\int_{ \Omega_+\setminus\Omega_+^{\varepsilon}} {\bf G_+}\times \nabla\varphi\,dx\\
&=I+II+III+IV
\end{align*}
$II,IV\to 0$ as $\varepsilon\to 0^+$ because the extensions $\G_\pm$ are locally bounded in $\Omega_\pm\cup \Gamma$ and $\varphi$ is compactly supported in $\Omega$. We write  $I=(I_1,I_2,I_3)$, ${\bf G}_-=(G_1^-,G_2^-,G_3^-)$, and ${\bf n^{\varepsilon}_{-}}=(n_1^{\varepsilon},n_2^{\varepsilon},n_3^{\varepsilon})$ the outward unit normal to $\partial \Omega^{\varepsilon}_-$. Using the divergence theorem
\begin{align*}
I_1&=\int_{\Omega_-^{\varepsilon}}G_2^-\varphi_{x_3}-G_3^-\varphi_{x_2}\, dx\\
&= \left(\int_{\partial \Omega_-^{\varepsilon}}\varphi G_2^- n_3^{\varepsilon}\, d\sigma-\int_{\Omega_-^{\varepsilon}}(G_2^-)_{x_3}\varphi\, dx\right)-\left(\int_{\partial \Omega_-^{\varepsilon}}\varphi G_3^- n_2^{\varepsilon}\, d\sigma-\int_{\Omega_-^{\varepsilon}}(G_3^-)_{x_2}\varphi\, dx\right)\\
&=\int_{\partial \Omega^{\varepsilon}_-}\varphi (G_2^- n_3^{\varepsilon}-G_3^-  n_2^{\varepsilon})\, d\sigma+\int_{\Omega^{\varepsilon}_{-}}\varphi((G_3^-)_{x_2}-(G_2^-)_{x_3})\, dx.
\end{align*}
Similarly
\begin{align*}
I_2&=\int_{\Omega_-^{\varepsilon}}-G_1^-\varphi_{x_3}+G_3^-\varphi_{x_1}\, dx
=\int_{\partial \Omega^{\varepsilon}_-}\varphi (G_3^- n_1^{\varepsilon}-G_1^-  n_3^{\varepsilon})\, d\sigma+\int_{\Omega^{\varepsilon}_{-}}\varphi((G_1^-)_{x_3}-(G_3^-)_{x_1})\, dx.
\end{align*}
and
\begin{align*}
I_3=\int_{\Omega_-^{\varepsilon}}G_1^-\varphi_{x_2}-G_2^-\varphi_{x_1}\, dx
=\int_{\partial \Omega^{\varepsilon}_-}\varphi (G_1^- n_2^{\varepsilon}-G_2^-  n_1^{\varepsilon})\, d\sigma+\int_{\Omega^{\varepsilon}_{-}}\varphi((G_2^-)_{x_1}-(G_1^-)_{x_2})\, dx.
\end{align*}
Combining the above calculations, we deduce that
$$I=\int_{\partial \Omega_-^{\varepsilon}}\varphi  {\bf G_{-}}\times {\bf n_-^{\varepsilon}} \, d\sigma+\int_{\Omega^{\varepsilon}_{-}}\varphi \nabla \times {\bf G_{-}}\, dx.$$
Since $\varphi$ is compactly supported in $\Omega$, ${\bf G_-}$ is locally bounded in $\Omega_-\cup \Gamma$ ,  and $\nabla \times {\bf G_-}\in L^1(\Omega_-)$ then 
 $$I\to \int_{\Gamma}  \varphi {\bf G}_{-}\times  {\bf n}\, d\sigma+\int_{\Omega_-}\varphi \nabla \times {\bf G_-}\, dx.$$
Similarly we get 
 $$III=\int_{\partial \Omega_+^{\varepsilon}}\varphi {\bf G_{+}}\times {\bf n}^{\varepsilon}_{+}\, d\sigma+\int_{\Omega^{\varepsilon}_{+}}\varphi \nabla \times {\bf G_{+}}\, dx\to \int_{\Gamma}  \varphi {\bf G}_{+}\times  ({\bf -n})\, d\sigma+\int_{\Omega_+}\varphi \nabla \times {\bf G_+}\, dx.$$
Hence \eqref{eq:formula for curl} follows.
\end{proof}

\subsection{Distributions depending on a parameter}\label{subsec:distributions depending on a parameter t}
Since the fields satisfying Maxwell's equations depend on time, 
we consider vector-valued distributions in $\Omega$ depending on a parameter $t\in \mathbb R$, that is, for each $t\in \mathbb R$, $g(\cdot,t)\in \mathcal D'(\Omega)$, see \cite[Appendix 2, p. 147]{gelfandshilov:distributions}. 
We need the following.
\begin{definition}\label{def:differentiability}Let $g(\cdot,t)$ be a distribution in $\Omega\subseteq\mathbb R^n$ depending on the parameter $t\in \R$.
We say that the derivative of $g(\cdot,t)$ with respect to the parameter $t$ exists if for each test function $\varphi\in \mathcal D(\Omega)$, the function $\langle g(\cdot,t),\varphi\rangle$ is differentiable in $t$, and there exists a distribution $h(\cdot,t)$ depending on the parameter $t$ such that 
$$\langle h(\cdot,t),\varphi\rangle=\dfrac{d}{dt}\langle g(\cdot,t),\varphi\rangle .$$
We write $h(x,t)=\dfrac{\partial g}{\partial t}(x,t).$
\end{definition}

\begin{proposition}\label{prop:switch derivative}
Given a distribution $g(\cdot,t)$ in $\Omega$, $t\in \mathbb R$, if $\dfrac{\partial g}{\partial t}(\cdot,t)$ exists for each $t\in \mathbb R$, then for every multi-index \cristian{$\alpha=(\alpha_1,\cdots ,\alpha_n)$ with $\alpha_i\in \Z_{\geq 0}$}, the derivative with respect to $t$ of the distribution $D^{\alpha} g$ exists and we have 
$$\dfrac{\partial (D^{\alpha} g)}{\partial t}=D^{\alpha}\left(\dfrac{\partial g}{\partial t}\right),$$
\cristian{with $D^\alpha=\dfrac{\partial^{\alpha_1}\partial^{\alpha_2}\cdots\partial^{\alpha_n}}{\partial x_1^{\alpha_1}\partial x_2^{\alpha_2}\cdots\partial x_n^{\alpha_n}}$.}
\end{proposition}

\begin{proof}
If $h(\cdot,t)=\dfrac{\partial g}{\partial t}$ and $\varphi\in \mathcal D(\Omega)$, then $\langle g(\cdot, t),\varphi\rangle$ is differentiable in $t$. 
Since
$$\langle D^{\alpha}g(\cdot,t),\varphi\rangle=(-1)^{|\alpha|}\langle g(\cdot,t),D^{\alpha}\varphi\rangle,$$ 
\cristian{with $|\alpha|=\alpha_1+\cdots +\alpha_n$,} and
$D^{\alpha}\varphi\in \mathcal D(\Omega)$, then $\langle D^\alpha g(\cdot,t),\varphi\rangle$ is differentiable in $t$ and
$$\dfrac{d}{dt}\langle D^{\alpha} g(\cdot,t),\varphi\rangle=(-1)^{|\alpha|}\dfrac{d}{dt}\langle g(\cdot,t),D^{\alpha}\varphi\rangle=(-1)^{|\alpha|} \langle h(\cdot ,t),D^{\alpha}\varphi\rangle=\langle D^{\alpha} h, \varphi\rangle.$$
\end{proof}

Recalling the set up at the beginning of this section in \cristian{Figure \ref{fig:interface_domain}},  $\Omega$ is a smooth open and bounded domain in $\mathbb R^3$, and $\Gamma$ is a smooth surface that splits $\Omega$ into two disjoint open parts $\Omega_+$ and $\Omega_-$, i.e., $\Omega=\Omega_-\cup \Gamma \cup \Omega_+$.
For $t\in\mathbb R$, we are given a function $g(\cdot,t)$ satisfying 
\begin{enumerate}[label=(H\arabic*)]
\item\label{condition 4} $g(\cdot,t)\in L^1_{loc}(\Omega)$ for every $t$,
\item\label{condition 5} for each fixed $x\in \Omega_\pm$ the function $g(x,\cdot)$ is differentiable with respect to $t$ and there exists a function $\psi\in L^1(\Omega_+\cup \Omega_-)$ such that $\left|\dfrac{\partial g}{\partial t}(x,t)\right|\leq \psi(x)$ for a.e. $x\in \Omega_\pm$ and for each $t$.
\end{enumerate}
For every $t$, the linear functional $g(\cdot,t)$ given by
$$\langle g(\cdot,t),\varphi\rangle=\int_{\Omega} g(x,t)\varphi(x)\,dx,$$
is then a well defined distribution by Item \ref{condition 4}.

\begin{proposition}\label{thm:representation formula for dg/dt}
\cristian{If $\G(\cdot,t)=(g_1(\cdot,t),g_2(\cdot ,t),g_3(\cdot ,t))$ is a vector-valued distribution with each $g_i$ satisfying \ref{condition 4} and \ref{condition 5}}, then 
the distribution $\G(\cdot,t)$ has a derivative with respect to the parameter $t$, and
$$\left\langle \dfrac{\partial \G}{\partial t}(\cdot, t),\varphi\right\rangle
=
\int_{\Omega_-} \dfrac{\partial \G}{\partial t}(x,t)\varphi(x)\, dx+\int_{\Omega_+} \dfrac{\partial \G}{\partial t}(x,t)\varphi(x)\, dx.$$
\end{proposition}

\begin{proof}
\cristian{Let us denote by $g(\cdot,t)$ any of the components of $\G$.} We write for $\varphi\in \mathcal D(\Omega)$ and $t\in \R$
\begin{align*}
\langle g(\cdot,t),\varphi\rangle
=
\int_{\Omega}g(x,t)\,\varphi(x)\,dx
=
\int_{\Omega_-}g(x,t)\,\varphi(x)\,dx+\int_{\Omega_+}g(x,t)\,\varphi(x)\,dx,
\end{align*}
since from \ref{condition 4}, the integral $\int_{\Gamma}g(x,t)\,\varphi(x)\,dx=0$.
Using condition \ref{condition 5} and the Lebesgue dominated convergence theorem, we can justify differentiation under the integral sign and obtain that
$\langle g(\cdot,t),\varphi\rangle$ is differentiable in $t$, and that
\begin{align*}
\dfrac{d}{dt}\langle g(\cdot,t),\varphi\rangle=
\int_{\Omega_-}\dfrac{\partial g}{\partial t}(x,t)\,\varphi(x)\,dx+\int_{\Omega_+}\dfrac{\partial g}{\partial t}(x,t)\,\varphi(x)\,dx.
\end{align*}
From \ref{condition 5}, the linear functional $h(\cdot,t)$ given by
$$\langle h(\cdot,t),\varphi\rangle=\int_{\Omega_-}\dfrac{\partial g}{\partial t}(x,t)\,\varphi(x)\,dx+\int_{\Omega_+}\dfrac{\partial g}{\partial t}(x,t)\,\varphi(x)\,dx,$$
is a well defined distribution and hence we obtain $h(\cdot,t)=\dfrac{\partial g}{\partial t}(\cdot,t)$.
\end{proof}

\setcounter{equation}{0}
\section{Maxwell equations in distributional sense and general boundary conditions}\label{sec:maxwell and boundary conditions}
We are given $\Omega$ open and bounded domain in $\mathbb R^3$, and $\Gamma$ a smooth surface separating $\Omega$ into two open parts $\Omega_+$ and $\Omega_-$ as in Section \ref{sec:preliminaries}.  We are interested in the Maxwell system \cite[Sections 1.1 and 1.2]{bornandwolf:principlesoptics} which written in Gaussian (or cgs) units has the form
\begin{equation}\label{eq:Maxwell system}
\begin{cases}
\nabla \times {\bf H}=\dfrac{4\pi}{c}{\bf J}+\dfrac{1}{c}\dfrac{\partial {\bf D}}{\partial t}\\
\nabla \cdot {\bf D}=4\pi \rho\\
\nabla \times {\bf E}=-\dfrac{1}{c}\dfrac{\partial {\bf B}}{\partial t}\\
\nabla \cdot {\bf B}=0
\end{cases},
\end{equation} 
where the curl and divergence are understood in the sense of distributions as in Section \ref{sec:preliminaries}, and the fields $\bf H, J, D, E, B$ are vector valued distributions in $\Omega$ depending on the parameter $t\in \R$ in the sense of Section \ref{subsec:distributions depending on a parameter t}, with $\J$ given, and $\rho$ is scalar distribution in $\Omega$, also given, depending also on the parameter $t\in \R$.

The purpose of this section is to show that under general assumptions on the current density field $\J$ and the charge density $\rho$ each equation in the Maxwell system \eqref{eq:Maxwell system}, understood in distributional sense, implies a boundary condition at the interface $\Gamma$ and 
the solutions are classical solutions away from $\Gamma$.
Viceversa, classical solutions in $\Omega_\pm$ discontinuous across $\Gamma$, give rise to distributions solutions in $\Omega$.
This is the contents of the following theorem.
\begin{theorem}\label{thm:general boundary conditions for Maxwell equations}
Let us assume that $\J$ and $\rho$ satisfy 
\begin{enumerate}
\item[(a)] $\J(x,t)=\J_0(x,t)+\nu_t$ with $\J_0(x,t)$ a locally integrable $\C^3$-valued function for $x\in \Omega$ for each $t$;
and $\nu_t$ is a family of $\C^3$-valued Borel measures in $\Omega$ depending on the parameter $t$ that are all concentrated on $\Gamma$, \cristian{that is, the support of $\nu_t$ is contained in $\Gamma$}; 
\item[(b)] 
$\rho(x,t)=\rho_0(x,t)+\mu_t$ with $\rho_0(x,t)$ locally integrable in $\Omega$ for each $t$, and $\mu_t$ are Borel measures in $\Omega$ depending on the parameter $t$ that are all concentrated on the surface $\Gamma$.
\end{enumerate}
Suppose also that $\B$ and $\D$ are given fields satisfying \ref{condition 1},  \ref{condition 2}, \ref{condition 3},  \ref{condition 4},  and \ref{condition 5}; 
and $\E$ and $\H$ are also given fields satisfying  \ref{condition 1},  \ref{condition 2}, and \ref{condition 3};  ${\bf n}$ denotes the unit normal to $\Gamma$ toward $\Omega_+$.

Then we have the following
\begin{enumerate}
\item If $\D$ satisfies $\nabla\cdot \D=4\pi\rho$ in $\Omega$ in the sense of distributions, then $\nabla \cdot \D_\pm(x,t)=4\pi\rho_0(x,t)$ for a.e. $x\in \Omega_\pm$ and for each $t$, and 
\begin{equation}\label{eq:density of mut bis}
d\mu_t(x)=\frac{1}{4\pi}[[\D(x,t)]]\cdot {\bf n}(x)\,d\sigma(x),
\end{equation}
where $d\sigma$ denotes the surface measure on $\Gamma$.
Reciprocally, if $\nabla\cdot \D_{\pm}=4\pi\rho_0$ holds point-wise in $\Omega_{\pm}$ and \eqref{eq:density of mut bis} holds, then $\D=\chi_{\Omega_-}\D_-+\chi_{\Omega_+}\D_+$ satisfies the equation $\nabla\cdot \D=4\pi\rho$ in $\Omega$ in the sense of distributions; as usual, $\chi_E$ denotes the characteristic function of the set $E$.
\item If $\B$ satisfies $\nabla \cdot \B=0$ in $\Omega$ in the sense of distributions, then $\nabla \cdot \B_\pm(x,t)=0$ point-wise $x\in \Omega_\pm$ and for each $t$ and \begin{equation}\label{eq:density of mut bis bis}
[[\B(x,t)]]\cdot {\bf n}(x)=0 \quad \text{for a.e. $x\in \Gamma$ (with respect to surface measure) for all $t$.}
\end{equation}
Reciprocally, if   $\nabla \cdot \B_{\pm}=0$ holds point-wise in $\Omega_{\pm}$ and \eqref{eq:density of mut bis bis} holds, then $\B=\chi_{\Omega_-}\B_-+\chi_{\Omega_+}\B_+$ satisfies the equation $\nabla \cdot \B=0$ in the sense of distributions.

\item If $\H$ and $\D$ satisfy $\nabla \times \H=\dfrac{4\pi}{c}\J+\dfrac{1}{c}\dfrac{\partial \D}{\partial t}$ in $\Omega$ in the sense of distributions, then 
$\nabla \times \H_\pm(x,t)=\dfrac{4\pi}{c}\J_0(x,t)+\dfrac{1}{c}\dfrac{\partial \D_\pm}{\partial t}(x,t)$ point-wise for $x\in \Omega_\pm$ and
\begin{equation}\label{eq:density of mut}
d\nu_t(x)=-\dfrac{c}{4\pi}[[\H(x,t)]]\times {\bf n}(x)\,d\sigma(x).
\end{equation}
Reciprocally, if the equation holds point-wise in $\Omega_\pm$ and 
\eqref{eq:density of mut} also holds, then the distributional equation holds for $\H=\chi_{\Omega_-}\H_-+\chi_{\Omega_+}\H_+$ and $\D=\chi_{\Omega_-}\D_-+\chi_{\Omega_+}\D_+$.

\item If $\E$ and $\B$ satisfy $\nabla \times \E=-\dfrac{1}{c}\dfrac{\partial \B}{\partial t}$ in $\Omega$ in the sense of distributions, then
$\nabla \times \E_\pm(x,t)=-\dfrac{1}{c}\dfrac{\partial \B_\pm}{\partial t}(x,t)$ point-wise in $\Omega_\pm$ and
\begin{equation}\label{eq:pointwise equation E cross n}
[[\E(x,t)]]\times {\bf n}(x)={\bf 0} \quad \text{for a.e. $x\in \Gamma$ (with respect to surface measure) for all $t$.}
\end{equation}
Reciprocally, if the equation holds point-wise in $\Omega_\pm$ and 
\eqref{eq:pointwise equation E cross n} also holds, then the distributional equation holds for $\E=\chi_{\Omega_-}\E_-+\chi_{\Omega_+}\E_+$ and $\B=\chi_{\Omega_-}\B_-+\chi_{\Omega_+}\B_+$.

\end{enumerate}

 \end{theorem}
\begin{proof}
(1) 
From $(b)$, $\rho(\cdot, t)$ is a distribution depending on $t$ given by
\[
\langle \rho(\cdot,t),\varphi\rangle
=
\int_{\Omega}\rho_0(x,t)\,\varphi(x)\,dx+\int_{\Omega}\varphi(x)\,d\mu_t(x)=
\int_{\Omega}\rho_0(x,t)\,\varphi(x)dx+\int_{\Gamma}\varphi(x)\,d\mu_t(x),
\]
for each $\p\in \mathcal D(\Omega)$;
and $\nabla\cdot \D$ is a distribution that acting on a test function $\varphi$ is given by \eqref{eq:formula for divergence}.
We have
\[
\langle \nabla\cdot \D(\cdot, t),\varphi \rangle=4\pi\langle \rho(\cdot ,t),\varphi \rangle\] 
for each $t$.
If $\text{supp}(\varphi )\subset \Omega_-$ or $\text{supp}(\varphi )\subset \Omega_+$,
then from \eqref{eq:formula for divergence}
\[
\int_{\Omega_-}\varphi(x)\, \nabla \cdot {\D_-(x,t)}\, dx+\int_{\Omega_+}\varphi(x)\, \nabla \cdot { \D_+(x,t)}\, dx
=4\pi\int_{\Omega}\rho_0(x,t)\,\varphi(x)\,dx,
\]
which implies
$\nabla \cdot \D_\pm(x,t)=4\pi\rho_0(x,t) \text{ for a.e. $x\in \Omega_\pm$ and for each $t$.}$
That is, the equation $\nabla\cdot \D=4\pi\rho$ is satisfied pointwise a.e. in $\Omega_\pm$.
If $\text{supp}(\varphi )\cap \Gamma\neq\emptyset$, we then get again 
from \eqref{eq:formula for divergence} 
that 
\[
\int_{\Gamma} \varphi(x)\, [[\D(x,t)]]\cdot {\bf n}(x)\, d\sigma(x)
=
4\pi \int_{\Gamma}\varphi(x)\,d\mu_t(x),
\]
that is, the measure $\mu_t$ has density $\dfrac{1}{4\pi}[[\D(x,t)]]\cdot {\bf n}(x)$, and \eqref{eq:density of mut bis}
follows. Notice that this part only uses that $\D$ satisfies \ref{condition 1}--\ref{condition 3}.

For the converse, applying \eqref{eq:formula for divergence} to $\D$ yields
\begin{align*}
\langle \nabla \cdot \D,\varphi\rangle
&=\int_{\Gamma} \varphi(x)\, [[\D(x)]]\cdot {\bf n}\, d\sigma
+\int_{\Omega_-}\varphi\, \nabla \cdot \D_-\, dx
+\int_{\Omega_+}\varphi \nabla \cdot \D_+\, dx\\
&=
4\pi\,
\int_{\Gamma} \varphi(x)\,d\mu_t
+4\pi\,\int_{\Omega_-}\varphi\,\rho_0\, dx
+\int_{\Omega_+}\varphi \,\rho_0\, dx\quad \text{from \eqref{eq:density of mut bis}}\\
&=
4\pi\,
\int_{\Gamma} \varphi(x)\,d\mu_t
+4\pi\,\int_{\Omega}\varphi\,\rho_0\, dx=4\pi\,\langle \rho,\p\rangle.
\end{align*}

(2) 
We proceed as in the proof of (1) and 
in this way we obtain
$\nabla \cdot \B_\pm(x,t)=0$ for a.e. $x\in \Omega_\pm$ and for each $t$;
and
 \eqref{eq:density of mut bis bis}.
 Reciprocally, $\B=\chi_{\Omega_-}\B_-+\chi_{\Omega_+}\B_+$ satisfies the equation $\nabla \cdot \B=0$ in the sense of distributions.

(3) The equation reads \[
\left\langle \nabla \times \H,\varphi \right\rangle 
=\dfrac{4\pi}{c}
\left\langle \J, \varphi \right\rangle
+
\dfrac{1}{c}\left\langle \dfrac{\partial \D}{\partial t}, \varphi \right\rangle,
\]
for each $\p\in \mathcal D(\Omega)$.
From Proposition \ref{thm:representation formula for dg/dt}
\[
\left\langle \dfrac{\partial \D}{\partial t}, \varphi \right\rangle
=
\int_{\Omega_-} \dfrac{\partial \D_-(x,t)}{\partial t}\varphi(x)\, dx+\int_{\Omega_+} \dfrac{\partial \D_+(x,t)}{\partial t}\varphi(x)\, dx.
\]
If $\text{supp}(\varphi )\subset \Omega_-$ or $\text{supp}(\varphi )\subset \Omega_+$,
then from \eqref{eq:formula for curl}
\begin{align*}
&\int_{\Omega_-}\varphi(x)\, \nabla \times \H_-(x,t)\, dx+\int_{\Omega_+}\varphi(x)\, 
\nabla \times \H_+(x,t)\, dx\\
&\qquad =\dfrac{4\pi}{c}\int_{\Omega}\J_0(x,t)\,\varphi(x)\,dx
+
\dfrac{1}{c}\int_{\Omega_-} \dfrac{\partial \D_-}{\partial t}(x,t)\varphi(x)\, dx+\dfrac{1}{c}\int_{\Omega_+} \dfrac{\partial \D_+}{\partial t}(x,t)\varphi(x)\, dx,
\end{align*}
which implies 
$\nabla \times \H_\pm(x,t)=\dfrac{4\pi}{c}\J_0(x,t)+\dfrac{1}{c}\dfrac{\partial \D_\pm}{\partial t}(x,t) \text{ for a.e. $x\in \Omega_\pm$ and all $t$}$.
If $\text{supp}(\varphi )\cap \Gamma\neq\emptyset$, we then get again from \eqref{eq:formula for curl} 
that 
\[
-\int_{\Gamma} \varphi(x)\, [[\H(x,t)]]\times  {\bf n}(x)\, d\sigma(x)
=
\dfrac{4\pi}{c}\int_{\Gamma}\varphi(x)\,d\nu_t(x),
\]
that is, the measure $\nu_t$ has density $-\dfrac{c}{4\pi}[[\H(x,t)]]\times {\bf n}(x)$, so
\eqref{eq:density of mut} follows.

If each equation $\nabla \times \H_\pm=\dfrac{4\pi}{c}\J_0+\dfrac{1}{c}\dfrac{\partial \D_\pm}{\partial t}$ holds in $\Omega_\pm$ in the classical sense and  \eqref{eq:density of mut} holds, then the equation $\nabla \times \H=\dfrac{4\pi}{c}\J+\dfrac{1}{c}\dfrac{\partial \D}{\partial t}$ is satisfied in $\Omega$ in the sense of distributions where $\H$ is the distribution given by the locally integrable function $\chi_{\Omega_-}\H_-+\chi_{\Omega_+}\H_+$.
In fact, from \eqref{eq:formula for curl} and \eqref{eq:density of mut}
\begin{align*}
\langle \nabla \times \H,\p \rangle
&=-\int_{\Gamma} \varphi(x)\,  [[\H(x)]]\times {\bf n}\, d\sigma
+\int_{\Omega_-}\varphi \,\nabla \times \H_-\, dx+\int_{\Omega_+}\varphi \,\nabla \times \H_+\, dx \\
&=
\dfrac{4\pi}{c}\int_{\Gamma} \varphi(x)\,  d\nu_t
+\dfrac{4\pi}{c}\int_{\Omega}\varphi \,\J_0\,dx
+
\dfrac{1}{c}\int_{\Omega_-}\varphi \,
\dfrac{\partial \D_-}{\partial t}\, dx
+
\dfrac{1}{c}\int_{\Omega_+}\varphi \,
\dfrac{\partial \D_+}{\partial t}\, dx\\
&=
\dfrac{4\pi}{c}\langle \J,\p\rangle+\dfrac{1}{c}\left\langle \dfrac{\partial \D}{\partial t},\p\right\rangle. 
\end{align*}

(4)
The equation reads \[
\left\langle \nabla \times \E,\varphi \right\rangle 
=
-\dfrac{1}{c}\left\langle \dfrac{\partial \B}{\partial t}, \varphi \right\rangle,
\]
for each $\p\in \mathcal D(\Omega)$.
From Proposition \ref{thm:representation formula for dg/dt}
\[
\left\langle \dfrac{\partial \B}{\partial t}, \varphi \right\rangle
=
\int_{\Omega_-} \dfrac{\partial \B(x,t)}{\partial t}\varphi(x)\, dx+\int_{\Omega_+} \dfrac{\partial \B(x,t)}{\partial t}\varphi(x)\, dx.
\]
If $\text{supp}(\varphi )\subset \Omega_-$ or $\text{supp}(\varphi )\subset \Omega_+$,
then from \eqref{eq:formula for curl}
\begin{align*}
&\int_{\Omega_-}\varphi(x)\, \nabla \times \E_-(x,t)\, dx+\int_{\Omega_+}\varphi(x)\, 
\nabla \times \E_+(x,t)\, dx\\
&\qquad =-\dfrac{1}{c}\int_{\Omega_-} \dfrac{\partial \B_-}{\partial t}(x,t)\varphi(x)\, dx-\dfrac{1}{c}\int_{\Omega_+} \dfrac{\partial \B_+}{\partial t}(x,t)\varphi(x)\, dx,
\end{align*}
which implies 
$\nabla \times \E_\pm(x,t)=-\dfrac{1}{c}\dfrac{\partial \B_\pm}{\partial t}(x,t)$ for a.e. $x\in \Omega_\pm$ and all $t$.
If $\text{supp}(\varphi )\cap \Gamma\neq\emptyset$, we then get again from \eqref{eq:formula for curl} that
\begin{align*}
\langle \nabla \times \E,\p \rangle
&=
-\int_{\Gamma} \varphi(x)\,  [[\E(x,t)]]\times {\bf n}\, d\sigma
+\int_{\Omega_-}\varphi \nabla \times \E_-\, dx+\int_{\Omega_+}\varphi \nabla \times \E_+\, dx\\
&=
-\int_{\Gamma} \varphi(x)\,  [[\E(x,t)]]\times {\bf n}\, d\sigma
-\dfrac{1}{c}\int_{\Omega_-}\varphi \dfrac{\partial \B_-}{\partial t}(x,t)\, dx
-\dfrac{1}{c}\int_{\Omega_+}\varphi \dfrac{\partial \B_+}{\partial t}(x,t)\, dx\\
&=
-\int_{\Gamma} \varphi(x)\,  [[\E(x,t)]]\times {\bf n}\, d\sigma
-\dfrac{1}{c}\left\langle \dfrac{\partial \B}{\partial t}, \varphi \right\rangle
\end{align*}
implying that 
\[\int_{\Gamma} \varphi(x)\,  [[\E(x,t)]]\times {\bf n}\, d\sigma=0\]
for all test functions $\p\in \mathcal D(\Omega)$ and all $t$, 
therefore \eqref{eq:pointwise equation E cross n} holds.
 
Reciprocally, if \eqref{eq:pointwise equation E cross n} holds we obtain that $\nabla \times \E=-\dfrac{1}{c}\dfrac{\partial \B}{\partial t}$ holds in $\Omega$ in the sense of distributions.

\end{proof}

\subsection{Compatibility condition} 
Let us assume that \eqref{eq:Maxwell system} holds with vector valued distributions $\E,\H,\D,\B,\J$ depending on the parameter $t$, with $\B,\D,\J$ also differentiable with respect to this parameter.
Hence from the first and second Maxwell equations in \eqref{eq:Maxwell system}, \eqref{eq:div of curl zero}, and Proposition \ref{prop:switch derivative} we have in the distributional sense the continuity equation

\cristian{\begin{equation}\label{pde in rho}
0=\nabla \cdot (\nabla \times {\bf H})=\dfrac{4\pi}{c}\nabla \cdot {\bf J}
+\dfrac{1}{c}\nabla \cdot \dfrac{\partial  {\bf D}}{\partial t}=\dfrac{4\pi}{c}\nabla \cdot {\bf J}+\dfrac{1}{c}\dfrac{\partial \rho}{\partial t}.
\end{equation}}

When $\J$ and $\rho$ satisfy the assumptions in Theorem \ref{thm:general boundary conditions for Maxwell equations}, equation \eqref{pde in rho} leads to the following compatibility condition between the current $\J$ and the density $\rho$:
\cristian{\[
\dfrac{4\pi}{c}\(\nabla \cdot \J_0+\nabla \cdot \nu_t\)
+
\dfrac{1}{c}\(\dfrac{\partial \rho_0}{\partial t}+\dfrac{\partial \mu_t}{\partial t}\)=0,
\]}
in the sense of distributions.

\setcounter{equation}{0}
\section{Generalized Snell's law deduced from Maxwell equations}\label{sec:generalized snell law proof}

Letting as usual \cite[Section 1.1.2]{bornandwolf:principlesoptics} the material or constitutive equations
\begin{equation}\label{eq:constitutive equations}
\D=\epsilon\,\E,\qquad \B=\mu\,\H,
\end{equation}
we obtain from \eqref{eq:Maxwell system}
\begin{align}
\nabla \cdot \epsilon \E&=4\,\pi\,\rho,\quad \label{gausslaw}\tag{M.1}\\
\nabla \cdot \mu\H&=0,\quad \label{gausslawmagnetic}\tag{M.2}\\
\nabla \times \E&=-\dfrac{\mu}{c}\dfrac{\partial \H}{\partial t},\quad \label{faradaylaw}\tag{M.3}\\
\nabla \times \H&=\dfrac{4\pi}{c} \,{\mathbf J} + \dfrac{\epsilon}{c} \dfrac{\partial \E}{\partial t}\label{amperemaxwelllaw}\tag{M.4},
\end{align}
where $\rho(x,t)$ is the charge density, $\J(x,t)$ is the current density vector, $\E$ is the electric field, $\H$ is the magnetic field,
and $\epsilon, \mu$ are constants, the permittivity and permeability of the media (isotropic), respectively. 
If $\epsilon_0, \mu_0$ are the permittivity and permeability of vacuum, then $\epsilon_{\text{rel}}=\epsilon/\epsilon_0$ and , $\mu_{\text{rel}}=\mu/\mu_0$ denote the relative permittivity and relative permeability, respectively, and $n=\sqrt{\epsilon_{\text{rel}} \mu_{\text{rel}}}$ is the refractive index of the media. Since the speed of light in vacuum is $c=1/\sqrt{\epsilon_0\mu_0}$, and the phase velocity of light in the media is $v=1/\sqrt{\epsilon \mu}$, then $v=c/n$.

Let $\Gamma$ be the plane $x_3=0$, and let $\Omega_+$, $\Omega_-$ denote the regions above $\Gamma$ and below $\Gamma$ respectively, with $\Omega_-$ filled with medium $I$ and $\Omega_+$ filled with medium $II$.
The constants $\epsilon,\mu$  in the Maxwell system \eqref{gausslaw}--\eqref{amperemaxwelllaw} may be different in media $I$ and $II$, and they are denoted by $\epsilon_-,\mu_-$ in medium $I$, and $\epsilon_+,\mu_+$ in medium $II$.
Suppose the incoming incident electric field in media $I$ is a plane wave with the form
%
\begin{equation}\label{eq:Ei}
\E_i(x,t)={\bf A}_i \,e^{\i\,\omega\(\frac{\k_i\cdot x}{v_1}-t\)}
\end{equation}
where $\k_i$ is the incident unit vector, $v_1$ is the velocity of propagation in medium $I$, ${\bf A}_i$ is a three dimensional constant complex vector, the amplitude, and $\omega$ is a constant (the angular frequency). Here $x=(x_1,x_2,x_3)$. \cristian{As usual, the Roman numeral $\rm i$ in the exponentials denotes the unit imaginary number.} This wave is defined for $x_3<0$, i.e., the field is incident to the plane $\Gamma$ from below and defined in $\Omega_-$.
This wave strikes the plane $\Gamma$ 
and it is then transmitted into medium $II$ as a nonlinear wave and {\it the ansatz is to assume it has the form}
\begin{equation}\label{eq:proposed form of the electric field}
\E_t(x,t)={\bf A}_t \,e^{\i\,\omega\(\frac{\k_t\cdot x}{v_2}+f(x)-t\)}
\end{equation}
where now $\k_t$ is the refracted unit vector, $v_2$ is the velocity of propagation in medium $II$, ${\bf A}_t$ is the amplitude, a constant vector, the wave is defined for $x_3>0$, i.e., on $\Omega_+$, and $f(x)$ is a $C^2$ function defined in a neighborhood of the plane $\Gamma$. 
There is also a wave reflected back into medium $I$ that will be assumed to have also a similar form
\begin{equation}\label{eq:proposed form of the electric field back}
\E_r(x,t)={\bf A}_r \,e^{\i\,\omega\(\frac{\k_r\cdot x}{v_1}+f(x)-t\)},
\end{equation}
with ${\bf A}_r$ a constant vector, $\k_r$ is the reflected back unit vector, $v_1$ is the velocity of propagation in medium $I$,  the wave is defined for $x_3<0$, i.e., on $\Omega_-$.
We are assuming that $f$ depends only on  $x_1,x_2$, i.e., the gradient of $f$ is tangential to the plane $\Gamma$; \cristian{we then denote $\nabla f=(f_{x_1},f_{x_2},0)$}.
In addition, and without loss of generality, we assume that $f(0)=0$, otherwise that simply changes the values of the amplitudes.


\begin{figure}[htbp]
\centering

\begin{tikzpicture}[scale=1.1, line cap=round, line join=round, >=Latex]

\definecolor{medone}{RGB}{90,140,220}
\definecolor{medtwo}{RGB}{240,170,90}
\definecolor{gammacol}{RGB}{180,30,40}

\def\L{2.2}

\fill[medone!25] (-5,0) rectangle (5,3);
\fill[medtwo!25] (-5,-3) rectangle (5,0);

\draw[very thick,gammacol] (-5,0) -- (5,0);

\fill (0,0) circle (1.4pt);
\node[above left] at (0,0) {$P$};

\draw[->,very thick] (0,0) -- (0,\L);
\node[right] at (0,\L) {$\nu=(0,0,1)$};

\draw[->,very thick] ({-0.92*\L},{-0.40*\L}) -- (0,0);
\node[left] at ({-0.55*\L},{-0.18*\L}) {$\k_i,\ (\E_i,\H_i)$};

\draw[->,very thick] (0,0) -- ({0.92*\L},{-0.40*\L});
\node[right] at ({0.62*\L},{-0.18*\L}) {$\k_r,\ (\E_r,\H_r)$};

\draw[->,very thick] (0,0) -- ({0.82*\L},{0.7*\L});
\node[right] at ({0.5*\L},{0.22*\L}) {$\k_t,\ (\E_t,\H_t)$};

\node at (3.7,1.9) {$\Omega_+$};
\node at (3.7,-1.9) {$\Omega_-$};

\node at (-3.6,1.9) {medium $II$};
\node at (-3.6,1.45) {$\epsilon_+,\mu_+$};

\node at (-3.6,-1.9) {medium $I$};
\node at (-3.6,-2.35) {$\epsilon_-,\mu_-$};

\node[gammacol, above right] at (-4.6,0) {$\Gamma=\{x_3=0\}$};

\end{tikzpicture}

\caption{Reflection and transmission at the planar interface $\Gamma=\{x_3=0\}$ separating medium $I$ from medium $II$}
\label{fig:planar-interface}
\end{figure}

The plan of this section is the following:
\begin{enumerate}
\item The fields $\E_i, \E_t, \E_r$ have corresponding magnetic fields $\H_i, \H_t, \H_r$ so that the Maxwell system \eqref{gausslaw}--\eqref{amperemaxwelllaw} is satisfied; these are calculated in Section \ref{subsect:calculation of magnetic fields}. These magnetic fields are used in Section \ref{sec:boundary conditions for magnetic fields} to obtain boundary conditions for them. 
\item Section \ref{subsect:main result and GSL} contains the proof of the main result, Theorem \ref{thm:MAIN THEOREM GSL}. 
The proof uses Theorem \ref{thm:general boundary conditions for Maxwell equations} to obtain the boundary conditions \eqref{eq:equation second components}, \eqref{eq:equation first components}, and  
\eqref{eq:equation third components} for the electric field (equations that will be used later in Section \ref{sec:calculation of the amplitudes of the reflected and transmitted fields BIS}).
As a consequence of these we obtain the generalized Snell law for the first two components of the wave vectors, Equation \eqref{eq:generalized snell law with psi}. 
\cristian{It also contains the statement of Lemma \ref{lm:exponential gradient} used to prove Theorem \ref{thm:MAIN THEOREM GSL} . The proof of this lemma is given in the Appendix \ref{app:proof of exponential lemma}}.
\item In Section \ref{subsec:GLS for general phases} we deduce from the previous items the generalized Snell law for refraction \eqref{eq:generalized snell law} and the generalized law of reflection \eqref{eq:generalized snell law reflection}.
\item 
In Section \ref{subsect:calculation of third components} we show relationships for the third components of the wave vectors, Corollary \ref{cor:third components}.
\item In Section \ref{rmk:orthogonality conditions} we deduce orthogonality conditions for the amplitudes that are used later in Section \ref{sec:calculation of the amplitudes of the reflected and transmitted fields BIS} to calculate the amplitudes of the transmitted and reflected waves.
\end{enumerate}

We use the following notation throughout the paper 
\begin{equation}\label{eq:notation for wave vectors}
\m_i=\omega\frac{\k_i}{v_1},\qquad \m_r=\omega \frac{\k_r}{v_1},\qquad \m_t=\omega \dfrac{\k_t}{v_2},
\end{equation} 
and write $\m_\ell=(m_1^\ell,m_2^\ell, m_3^\ell)$, 
with $\ell=i,r,t$.

\subsection{Calculation of the corresponding magnetic fields}\label{subsect:calculation of magnetic fields}
The values of these magnetic fields are given in the following lemma.

\begin{lemma}\label{eq:corresponding magnetic fields}
Suppose the electric fields $\E_i, \E_t, \E_r$ are given by \eqref{eq:Ei}, \eqref{eq:proposed form of the electric field}, and \eqref{eq:proposed form of the electric field back}, respectively. If the field $\H'(x,t)$ solves 
\[
\nabla \times \(\E_i+\E_r\)=-\dfrac{\mu_-}{c} \dfrac{\partial \H'}{\partial t}
\]
in $\Omega_-$, and the field $\H_t(x,t)$ solves 
\[
\nabla \times \E_t=-\dfrac{\mu_+}{c} \dfrac{\partial \H_t}{\partial t}
\]
in $\Omega_+$, then $\H'=\H_i+\H_r$ with 
\[
\H_i=-\dfrac{c}{\mu_-}\,\E_i\times \dfrac{\k_i}{v_1};\quad
\H_r
=-\dfrac{c}{\mu_-}\,\E_r\times \(\dfrac{\k_r}{v_1}+\nabla f(x)\),
\]
and
\[
\H_t=-\dfrac{c}{\mu_+}\,\E_t\times \(\dfrac{\k_t}{v_2}+\nabla f(x)\),
\]
modulo fields only depending on $x$ which we assume to be zero. We have denoted $\nabla f(x)=\(f_{x_1}(x),f_{x_2}(x),0\)$.

\end{lemma}

\begin{proof}

Calculation of $\H'$. 
We have from \eqref{eq:Ei} and \eqref{eq:proposed form of the electric field back}\[
\nabla \times \(\E_i +\E_r\)
=-\i\,\omega\,\({\bf A}_i\times \dfrac{\k_i}{v_1}\)\,e^{\i\,\omega\(\frac{\k_i\cdot x}{v_1}-t\)}
-\i\,\omega\,{\bf A}_r\times \(\dfrac{\k_r}{v_1}+\nabla f(x)\)\,e^{\i\,\omega\(\frac{\k_r\cdot x}{v_1}+f(x)-t\)}.
\]
Integrating yields
\begin{align*}
\H'&=-\dfrac{c}{\mu_-}\int \nabla \times \(\E_i+\E_r\)\,dt\\
&=
-\dfrac{c}{\mu_-}\,\({\bf A}_i\times \dfrac{\k_i}{v_1}\)\,e^{\i\,\omega\(\frac{\k_i\cdot x}{v_1}-t\)}
-\dfrac{c}{\mu_-}\,{\bf A}_r\times \(\dfrac{\k_r}{v_1}+\nabla f(x)\)\,
e^{\i\,\omega\(\frac{\k_r\cdot x}{v_1}+f(x)-t\)}\\
&=
-\dfrac{c}{\mu_-}\,\E_i\times \dfrac{\k_i}{v_1}
-\dfrac{c}{\mu_-}\,\E_r\times \(\dfrac{\k_r}{v_1}+\nabla f(x)\),
\end{align*}
modulo a field only depending on $x$ which we assume to be zero.

Similarly, the calculation of $\H_t$ follows by integration using \eqref{eq:proposed form of the electric field}.

\end{proof}

\subsection{Main result and the generalized Snell law}\label{subsect:main result and GSL}
\cristian{In this section, we shall prove Theorem \ref{thm:MAIN THEOREM GSL} showing that the phase function $f$ in the scattered waves \eqref{eq:proposed form of the electric field} and \eqref{eq:proposed form of the electric field back} is necessarily affine and we prove relationships between the components of the wave vectors $\k_i,\k_r$ and $\k_t$ and $\nabla f$, that imply the generalized Snell law \eqref{eq:generalized snell law} and \eqref{eq:generalized snell law reflection}. The proof of Theorem \ref{thm:MAIN THEOREM GSL}
requires the following lemma whose proof is given in the Appendix \ref{app:proof of exponential lemma}.
\begin{lemma}\label{lm:exponential gradient}
Suppose that the equation 
\begin{equation}\label{eq:general exponential equation a b c}
a\,e^{\i\,\m_i\cdot X}+b\,e^{\i\,\left(\m_r\cdot X+\psi(X)\right)}
+c\,e^{\i\,\left(\m_t\cdot X+\psi(X)\right)}=0
\end{equation}
holds for all $X=(x_1,x_2,0)$, where $a,b,c$ are fixed complex constants,
$\m_\ell=\(m_1^\ell,m_2^\ell,m_3^\ell\)$ for $\ell=i,r,t$, and $\psi$ is a real-valued $C^2$-function; we use the notation $\psi(X)=\psi(x_1,x_2)$.
We have
\begin{enumerate}
\item[(i)] If $a\neq 0$, $b\neq 0$, and $c\neq 0$, then $\psi$ is an affine function and
\[
m_j^i-m_j^r=\psi_{x_j}, \text{ and } m_j^i-m_j^t=\psi_{x_j}, \quad j=1,2.
\]
\item[(ii)] If $a=0$, $b\neq 0$, and $c\neq 0$, then $m_j^r=m_j^t$ for $j=1,2$.
\item[(iii)] If $a\neq 0$, $b=0$, and $c\neq 0$, then $\psi$ is affine and 
\[
m_j^i-m_j^t=\psi_{x_j}, \quad j=1,2.
\]
\item[(iv)] If $a\neq 0$, $b\neq 0$, and $c= 0$, then $\psi$ is affine and 
\[
m_j^i-m_j^r=\psi_{x_j}, \quad j=1,2.
\]
\end{enumerate}
\end{lemma}
Now, let us proceed to state and the prove the main result of the section.
}

\begin{theorem}\label{thm:MAIN THEOREM GSL}
Recall the definitions of the fields $\E_i,\E_r$ and $\E_t$ from \eqref{eq:Ei}--\eqref{eq:proposed form of the electric field back}, and the corresponding magnetic fields $\H_i,\H_r$ and $\H_t$ from Lemma \ref{eq:corresponding magnetic fields}.
Let 
\begin{equation}\label{def:definition of primed fields E' H'}
\E'(x,t)=\E_i(x,t)+\E_r(x,t),\quad \H'(x,t)=\H_i(x,t)+\H_r(x,t),
\end{equation}
and define $\E(x,t)=\chi_{\Omega_-} \E'(x,t)+\chi_{\Omega_+} \E_t(x,t)$, and $\H(x,t)= \chi_{\Omega_-} \H'(x,t)+\chi_{\Omega_+} \H_t(x,t)$.

Then the fields $\E',\H', \E_t$ and $\H_t$ satisfy conditions \ref{condition 1},  \ref{condition 2}, and \ref{condition 3}.

If $\E$ and $\H$ are distributional solutions to \eqref{faradaylaw},
and $\E$ is a distributional solution to  \eqref{gausslaw} with $\rho$ non-singular, \cristian{i.e. $\rho$ is as in Theorem \ref{thm:general boundary conditions for Maxwell equations}(b) with $\mu_t=0$}, then
the function $f$ must be affine, i.e., a linear function of $x_1,x_2$ plus a constant, and 
\begin{equation}\label{eq:generalized snell law with psi} 
\frac{k_j^i}{v_1}-\frac{k_j^r}{v_1}=f_{x_j} \quad \text{and}\quad  \frac{k_j^i}{v_1}-\frac{k_j^t}{v_2}=f_{x_j} \quad \quad j=1,2,
\end{equation}
with $\k_\ell=\(k_1^\ell,k_2^\ell,k_3^\ell\)$ with $\ell=i,r,t$.

\end{theorem}

\begin{proof}

The proof uses Lemma \ref{lm:exponential gradient} below and Theorem \ref{thm:general boundary conditions for Maxwell equations}.

From the explicit form of the fields, it is clear that $\E',\H', \E_t$ and $\H_t$ satisfy conditions \ref{condition 1},  \ref{condition 2}, and \ref{condition 3}. 
Since we assume that $\E$ and $\H$ are distributional solutions to \eqref{faradaylaw},
it follows that Theorem \ref{thm:general boundary conditions for Maxwell equations}, Part (4), is applicable.

Therefore, the jump of the electric field equals
\begin{align*}
[[\E(X,t)]]
&=
\lim_{x\to X,x\in \Omega_+}\E(x,t)-\lim_{x\to X,x\in\Omega_-}\E(x,t)\\ 
&=
{\bf A}_t \,e^{\i\,\omega\(\frac{\k_t\cdot X}{v_2}+f(X)-t\)}
-
{\bf A}_i \,e^{\i\,\omega\(\frac{\k_i\cdot X}{v_1}-t\)}
-
{\bf A}_r \,e^{\i\,\omega\(\frac{\k_r\cdot X}{v_1}+f(X)-t\)}\,
\end{align*}
with $X=(x_1,x_2,0)$,
and from the boundary condition \eqref{eq:pointwise equation E cross n}
\begin{equation}\label{eq:boundary condition with phi}
\({\bf A}_t\times \n\) \,e^{\i\,\omega\(\frac{\k_t\cdot X}{v_2}+f(X)-t\)}
-
\({\bf A}_i\times \n\) \,e^{\i\,\omega\(\frac{\k_i\cdot X}{v_1}-t\)}
-
\({\bf A}_r\times \n\) \,e^{\i\,\omega\(\frac{\k_r\cdot X}{v_1}+f(X)-t\)}={\bf 0},
\end{equation}
for all $X\in \Gamma$.

If we write the components of $A_\ell=\(A_1^\ell,A_2^\ell,A_3^\ell\)$ with $\ell=i,r,t$, then
$
A_\ell\times \n=\(A_2^\ell,-A_1^\ell,0\)$.
Also, if we set $\psi(X)=\omega f(X)$ and recall \eqref{eq:notation for wave vectors}, then 
\eqref{eq:boundary condition with phi} is the system of two scalar equations
\begin{equation}\label{eq:equation second components}
-A_2^i \,e^{\i\,\(\m_i\cdot X\)}
-A_2^r \,e^{\i\,\(\m_r\cdot X+\psi(X)\)}
+
A_2^t \,e^{\i\,\(\m_t\cdot X+\psi(X)\)}=0,
\end{equation}
\begin{equation}\label{eq:equation first components}
A_1^i \,e^{\i\,\(\m_i\cdot X\)}
+A_1^r \,e^{\i\,\(\m_r\cdot X+\psi(X)\)}
-
A_1^t \,e^{\i\,\(\m_t\cdot X+\psi(X)\)}=0.
\end{equation}
To prove the desired result we shall use Lemma \ref{lm:exponential gradient}. Let us first assume that
\begin{equation}\label{eq:Ai not parallel to the normal}
{\bf A}_i\times \n\neq 0,\quad {\bf A}_r\times \n\neq 0, \quad \text{and } {\bf A}_t\times \n\neq 0.
\end{equation}

From \eqref{eq:Ai not parallel to the normal} we have that $A_2^\ell \neq 0$ or $A_1^\ell\neq 0$ for $\ell=i, r,t$. 
Notice that in \eqref{eq:equation second components} if one coefficient is different from zero then at least one of the other two must be different from zero; and likewise in \eqref{eq:equation first components}. If in \eqref{eq:equation second components} $A_2^\ell\neq 0$ for $\ell=i,r,t$, then applying Lemma \ref{lm:exponential gradient} (i) it follows that $\psi$ is affine and \eqref{eq:generalized snell law with psi} holds.
Likewise, if in \eqref{eq:equation first components} $A_1^\ell\neq 0$ for $\ell=i,r,t$, then applying Lemma \ref{lm:exponential gradient} (i) it follows that $\psi$ is affine and \eqref{eq:generalized snell law with psi} holds.
If in \eqref{eq:equation second components} $A_2^i\neq 0$, then $A_2^r\neq 0$ or $A_2^t\neq 0$. If $A_2^r\neq 0$ and $A_2^t=0$, by Lemma \ref{lm:exponential gradient} (iv) we have $\psi$ is affine and $m_j^i-m_j^r=\psi_{x_j}$, $j=1,2$.
Since ${\bf A}_t\times \n\neq 0$, if $A_2^t=0$ we then must have $A_1^t\neq 0$ and from \eqref{eq:equation first components} $A_1^i\neq 0$ or $A_1^r\neq 0$. If $A_1^i\neq 0$ and $A_1^r= 0$, then by Lemma \ref{lm:exponential gradient} (iii) $m_j^i-m_j^t=\psi_{x_j}$ for $j=1,2$; so \eqref{eq:generalized snell law with psi} follows. 
On the other hand, if $A_1^i= 0$ and $A_1^r\neq 0$, then by Lemma \ref{lm:exponential gradient} (ii) $m_j^r=m_j^t$ for $j=1,2$ and so also \eqref{eq:generalized snell law with psi} follows.
In general, all the possibilities for the values of the coefficients with their conclusions are summarized in the following table:

\begin{center}
        \begin{tabular}{ | c | c | c | c | c | c | c |}
        \hline
            $A_2^i$ & $A_2^r$ & $A_2^t$ & $A_1^i$ & $A_1^r$ & $A_1^t$  & Conclusion \\
            \hline
            $\neq 0$ & $\neq 0$ & 0 & 0 & $\neq 0$ & $\neq 0$ & $\m_i-\m_r=\nabla \psi,$ and $\m_r=\m_t$ by Lemma \ref{lm:exponential gradient} (iv) (ii)\\
            0 & $\neq 0$ & $\neq 0$ & $\neq 0$ & 0 & $\neq 0$ & $\m_r=\m_t$ and $\m_i-\m_t=\nabla \psi,$ by Lemma \ref{lm:exponential gradient} (ii) (iii)\\
            $\neq 0$ & 0 & $\neq 0$ & 0 & $\neq 0$ & $\neq 0$ &  $\m_i-\m_t=\nabla \psi,$ and $\m_r=\m_t$ by Lemma \ref{lm:exponential gradient} (iii) (ii)\\ 
            $\neq 0$ & 0 & $\neq 0$ & $\neq 0$ & $\neq 0$ & 0 &  $\m_i-\m_t=\nabla \psi,$ and $\m_i-\m_r=\nabla \psi$ by Lemma \ref{lm:exponential gradient} (iii) (iv)\\
            0 & $\neq 0$ & $\neq 0$ & $\neq 0$ & $\neq 0$ & 0 &  $\m_r=\m_t,$ and $\m_i-\m_r=\nabla \psi$ by Lemma \ref{lm:exponential gradient} (ii) (iv)\\  
            0 & $\neq 0$ & $\neq 0$ & $\neq 0$ & 0 & $\neq 0$ &  $\m_r=\m_t,$ and $\m_i-\m_t=\nabla \psi$ by Lemma \ref{lm:exponential gradient} (ii) (iii)\\        \hline
        \end{tabular}
    \end{center}

Therefore $\psi$ is affine and \eqref{eq:generalized snell law with psi} follows when \eqref{eq:Ai not parallel to the normal} holds since $\psi=\omega f$.

It remains to prove \eqref{eq:generalized snell law with psi} when \eqref{eq:Ai not parallel to the normal} does not hold. 
That is, suppose 
\begin{equation}\label{eq:at least one Ai is parallel to the normal}
{\bf A}_i\times \n= 0,\text{ or } {\bf A}_r\times \n= 0, \text{ or }  {\bf A}_t\times \n= 0.
\end{equation}
Here we use the constitutive equations \eqref{eq:constitutive equations} and Part (1) of Theorem \ref{thm:general boundary conditions for Maxwell equations}. 
Notice that the permittivity constant for $\Omega_-$ is $\epsilon_-$ and for 
$\Omega_+$ is $\epsilon_+$. 
We recall the assumption that the field ${\bf D}=\epsilon_-\,\chi_{\Omega_-}\({\bf E}_i+{\bf E}_r\)+\epsilon_+\,\chi_{\Omega_+}{\bf E}_t=
\epsilon_-\,\chi_{\Omega_-}\E'+\epsilon_+\,\chi_{\Omega_+}{\bf E}_t$ is a distributional solution to the second equation in \eqref{eq:Maxwell system} with $\rho$ having singular part equals zero. Then Part (1) of Theorem \ref{thm:general boundary conditions for Maxwell equations} is applicable with $\mu_t=0$ and we have 
\begin{equation}\label{eq:boundary condition for D normal n}
[[{\bf D}(X,t)]]\cdot {\bf n}=0,
\end{equation}
where 
\[
[[{\bf D}(X,t)]]
=\lim_{x\to X,x\in \Omega_+}\epsilon_+\E_t(x,t)-\lim_{x\to X,x\in\Omega_-}\epsilon_-\E'(x,t). 
\]
Since ${\bf n}=(0,0,1)$, we then have from the form of the fields $\E_t$ and $\E'$ that the equation \eqref{eq:boundary condition for D normal n} reads ($\psi=\omega f$)
\begin{equation}\label{eq:equation third components}
\epsilon_-\,A_3^ie^{\i\,\m_i\cdot X}+\epsilon_-\,A_3^re^{\i\,\left(\m_r\cdot X+\psi(X)\right)}-\epsilon_+\,A_3^te^{\i\,\left(\m_t\cdot X+\psi(X)\right)}=0,
\end{equation}
which will be used to deal with the case \eqref{eq:at least one Ai is parallel to the normal}.
To begin with let us assume ${\bf A}_i\times \n=0$, that is, $A_1^i=0$ and $A_2^i=0$. Since ${\bf A}_i\neq 0$, it follows that $A_3^i \neq 0$.
Hence from \eqref{eq:equation third components} it follows that $A_3^r\neq 0$ or $A_3^t\neq 0$.
If $A_3^r\neq 0$ and $A_3^t\neq 0$, then by Lemma \ref{lm:exponential gradient} (i) we get that $\psi$ is affine and \eqref{eq:generalized snell law with psi} holds.
If $A_3^i\neq 0$, $A_3^r\neq 0$ and $A_3^t=0$, by Lemma \ref{lm:exponential gradient} (iv) it follows that $m_j^i-m_j^r=\psi_{x_j}$ for $j=1,2$.
But if $A_3^t=0$, since the amplitude ${\bf A}_t\neq 0$, we must have $A_1^t\neq 0$ or $A_2^t\neq 0$. 
If $A_1^t\neq 0$ and $A_2^t= 0$, then using \eqref{eq:equation first components}
we must have $A_1^r\neq 0$, and by Lemma \ref{lm:exponential gradient} (ii) $m_j^r=m_j^t$ for $j=1,2$ and so 
\eqref{eq:generalized snell law with psi} follows.
If $A_1^t= 0$ and $A_2^t\neq 0$, then from \eqref{eq:equation second components} $A_2^r\neq 0$ so by  Lemma \ref{lm:exponential gradient} (ii) $m_j^r=m_j^t$ for $j=1,2$ and so 
\eqref{eq:generalized snell law with psi} follows.

If $A_3^i\neq 0$, $A_3^r= 0$ and $A_3^t\neq 0$, by Lemma \ref{lm:exponential gradient} (iii) it follows that $m_j^i-m_j^t=\psi_{x_j}$ for $j=1,2$.
But if $A_3^r=0$, since the amplitude ${\bf A}_r\neq 0$, we must have $A_1^r\neq 0$ or $A_2^r\neq 0$. 
If $A_1^r\neq 0$ and $A_2^r= 0$, then using \eqref{eq:equation first components}
we must have $A_1^t\neq 0$, and by Lemma \ref{lm:exponential gradient} (ii) $m_j^r=m_j^t$ for $j=1,2$ and so 
\eqref{eq:generalized snell law with psi} follows.
If $A_1^r= 0$ and $A_2^r\neq 0$, then from \eqref{eq:equation second components} $A_2^t\neq 0$ so by  Lemma \ref{lm:exponential gradient} (ii) $m_j^r=m_j^t$ for $j=1,2$ and so 
again \eqref{eq:generalized snell law with psi} follows.

The remaining cases in \eqref{eq:at least one Ai is parallel to the normal} are treated similarly.
\end{proof}

\subsection{Deduction of the generalized Snell law for a general phase discontinuity}\label{subsec:GLS for general phases}
Let us now consider a general phase discontinuity function $\phi$ defined on the plane $x_3=0$ and let $\E_i$ be the incident electric field given by \eqref{eq:Ei}. For a point $P$ on the plane $x_3=0$, we model the scattering of the wave taking into account the value of the gradient of $\phi$, in other words, the function $f$ in \eqref{eq:proposed form of the electric field} and \eqref{eq:proposed form of the electric field back} is chosen to be 
$f(x)=\nabla \phi(P)\cdot x$\cristian{; $x=(x_1,x_2)$.} 
We will prove in Section \ref{sec:calculation of the amplitudes of the reflected and transmitted fields BIS} that the amplitudes of the scattered waves \eqref{eq:proposed form of the electric field} and \eqref{eq:proposed form of the electric field back} depend on the choice of the function $f$ and therefore in this case will depend of the point $P$.
Applying Theorem \ref{subsect:main result and GSL} with this choice of $f$ we then obtain from \eqref{eq:generalized snell law with psi} the generalized Snell's law for refraction \eqref{eq:generalized snell law} and the generalized law of reflection \eqref{eq:generalized snell law reflection}.
\cristian{In fact, since $n_i=c/v_i$, from \eqref{eq:generalized snell law with psi} we have $n_1 k_j^i-n_2k_j^t=cf_{x_j}$ for $j=1,2$ and since $\n=(0,0,1)$ taking $\lambda=n_1k_3^i-n_2k_3^t$, \eqref{eq:generalized snell law} follows with $f$ replaced by $cf$. Similarly, \eqref{eq:generalized snell law reflection} follows with $\lambda'=n_1k_3^i-n_1k_3^r$.} 

We can now extend this argument as follows. Suppose we take a finite number of points $P_1,\cdots ,P_N$ on the plane $x_3=0$, and for each point $P_j$ we choose the phase function $\nabla \phi(P_j)\cdot x$. Each of these phases gives rise to a transmitted and a reflected back waves. Then by superposition, the scattered waves for all the points $P_j$ will have the form
\[
\sum_{j=1}^N {\bf A}_t(P_j) \,e^{\i\,\omega\(\frac{\k_t\cdot x}{v_2}+\nabla \phi(P_j)\cdot x-t\)},
\]  
for the transmitted wave and
\[
\sum_{j=1}^N {\bf A}_r(P_j) \,e^{\i\,\omega\(\frac{\k_r\cdot x}{v_1}+\nabla \phi(P_j)\cdot x-t\)}
\]
for the wave reflected back.

\subsection{Calculation of the third components of the wave vectors}\label{subsect:calculation of third components}
Recalling the notation \eqref{eq:notation for wave vectors}, and from 
\eqref{eq:generalized snell law with psi} we have
\[
m_j^i-m_j^r=\omega f_{x_j},\quad m_j^i-m_j^t=\omega f_{x_j}\qquad j=1,2.
\]
The following corollary, shows formulas for the third component $m_3^{\ell}$ for $\ell=r,t$ as a function of $m_1^i$, $m_2^i$ and $f$.  

\begin{corollary}\label{cor:third components}
Under the assumptions of Theorem \ref{thm:MAIN THEOREM GSL}, and if 
\begin{equation}\label{eq:compatibility conditions for mi and f}
\sqrt{\left(m_1^i-\omega f_{x_1}\right)^2+\left(m_2^i-\omega f_{x_2}\right)^2}\leq \min\left\{\frac{\omega}{v_1},\frac{\omega}{v_2} \right\},
\end{equation}
then
\begin{align}
m_3^r=-\sqrt{\left(\frac{\omega}{v_1}\right)^2-\left(m_1^i-\omega f_{x_1}\right)^2-\left(m_2^i-\omega f_{x_2}\right)^2}\label{eq:m3r}\\
m_3^t=\sqrt{\left(\frac{\omega}{v_2}\right)^2-\left(m_1^i-\omega f_{x_1}\right)^2-\left(m_2^i-\omega f_{x_2}\right)^2}.\label{eq:m3t}
\end{align}
\end{corollary}

\begin{proof}
%

From \eqref{eq:generalized snell law with psi}
\begin{equation}\label{eq:relations between mi, mr}
\m_i-\m_r=\lambda \,\n +\omega \(f_{x_1},f_{x_2},0\),
\end{equation}
where ${\bf n}$ is the vertical unit direction, and $\lambda=m_3^i-m_3^r$. 
Dotting \eqref{eq:relations between mi, mr} with $\m_i$ and with $\m_r$ yields
\begin{align*}
\m_i\cdot \m_i-\m_r\cdot \m_i&=\lambda  m_3^i +\omega \(f_{x_1},f_{x_2},0\)\cdot \m_i\\
\m_i\cdot \m_r-\m_r\cdot \m_r&=\lambda m_3^r+\omega \(f_{x_1},f_{x_2},0\)\cdot \m_r,
\end{align*}
and adding these equations, we obtain that
\[
|\m_i|^2-|\m_r|^2=\lambda \, (m_3^i+m_3^r)+\omega \((m_1^i+m_1^r)f_{x_1}+(m_2^i+m_2^r)f_{x_2}\).
\]
Since $|\m_i|^2=|\m_r|^2=\omega^2/v_1^2$, it follows that 
\begin{equation}\label{eq:equation for lambda reflected back}
\lambda\,\(m_3^i+m_3^r\)+\omega \((m_1^i+m_1^r)f_{x_1}+(m_2^i+m_2^r)f_{x_2}\)=0.
\end{equation}
Hence, from \eqref{eq:relations between mi, mr}
\[
\(m_3^i\)^2-\(m_3^r\)^2+\omega \((m_1^i+m_1^r)f_{x_1}+(m_2^i+m_2^r)f_{x_2}\)=0.
\]
Since $\k_r$ is the unit direction of the ray reflected back in medium $I$, we have $m_3^r=\dfrac{\omega}{v_1}k_3^r<0$ and so solving for $m_3^r$ gives
\[
m_3^r=-\sqrt{\(m_3^i\)^2+\omega \((m_1^i+m_1^r)f_{x_1}+(m_2^i+m_2^r)f_{x_2}\)}.
\]
Since $\k_{i}$ is a unit vector with $k_3^i>0$, it follows that $m_3^i=\sqrt{\(\omega/v_1\)^2-(m_1^i)^2-(m_2^i)^2}$
%
and from the fact that $m_j^r=m_j^i-\omega f_{x_j}$ for $j=1,2$ we get
\begin{align*}
m_3^r
&=
- \sqrt{\left(\frac{\omega}{v_1}\right)^2-(m_1^i)^2-(m_2^i)^2+\omega \((2m_1^i-\omega f_{x_1})f_{x_1}+((2m_2^i-\omega f_{x_2})f_{x_2}\)}\\
&=
- \sqrt{\left(\frac{\omega}{v_1}\right)^2-\(m_1^i-\omega f_{x_1}\)^2-\(m_2^i-\omega f_{x_2}\)^2},
\end{align*}
which proves \eqref{eq:m3r}.

To prove \eqref{eq:m3t},
we proceed similarly using \eqref{eq:generalized snell law with psi} and that  $\m_t=\dfrac{\omega}{v_2}\k_t$, $|\k_t|=1$,  with  $m_3^t>0$.
\end{proof}

\begin{remark}\rm 
In the standard case, i.e., when $f=0$, 
\eqref{eq:compatibility conditions for mi and f} obviously reads $\sqrt{(m_1^i)^2+(m_2^i)^2}\leq \min\left\{\omega/v_1,\omega/v_2 \right\}$.
In particular,
when the refractive index of medium $I$ is smaller than the one for medium $II$, that is, when $v_1>v_2$, this always happens since $|\m_i|=\omega/v_1$.
On the other hand, if $v_1<v_2$, given an incident wave satisfying \eqref{eq:compatibility conditions for mi and f} to avoid total internal reflection, the third component must satisfy
\[
\k_i\cdot {\bf n}=\dfrac{v_1}{\omega}m_3^i=\dfrac{v_1}{\omega}\sqrt{\left(\dfrac{\omega}{v_1}\right)^2 - (m_1^i)^2 - (m_2^i)^2}\geq \dfrac{v_1}{\omega}\sqrt{\left(\dfrac{\omega}{v_1}\right)^2 - \left(\dfrac{\omega}{v_2}\right)^2}=\sqrt{1 - \left(\dfrac{v_1}{v_2}\right)^2}.
\]
This condition agrees with the one obtained for the standard Snell's law, see \cite[Section 2]{2015-gutierrez-sabra:asphericallensdesignandimaging}.
On the other hand, when $f \neq 0$, the compatibility conditions for $\m_i$ and $f$ in \eqref{eq:compatibility conditions for mi and f} must be satisfied in order to have reflected and transmitted waves.

\end{remark}
We also remark that from \eqref{eq:m3r} and \eqref{eq:m3t} the following relation between the third components of $\m_r$ and $\m_t$ holds:
\begin{equation}\label{eq:formula for m3t in terms of m3r}
m_3^t=\sqrt{\omega^2/v_2^2-\omega^2/v_1^2+\(m_3^r\)^2}.
\end{equation}

\subsection{Orthogonality conditions for the amplitudes}\label{rmk:orthogonality conditions} 
The following conditions must be satisfied by the amplitudes, which will be utilized in Section \ref{sec:calculation of the amplitudes of the reflected and transmitted fields BIS}.

\begin{lemma}\label{lm:orthogonality conditions for amplitudes}
Recall from \eqref{def:definition of primed fields E' H'} the definitions of $\E'$ and $\H'$ and 
the fields $\E(x,t)=\chi_{\Omega_-} \E'(x,t)+\chi_{\Omega_+} \E_t(x,t)$, and $\H(x,t)= \chi_{\Omega_-} \H'(x,t)+\chi_{\Omega_+} \H_t(x,t)$.

If $\E$ and $\H$ are distributional solutions to \eqref{faradaylaw},
and $\E$ is a distributional solution to  \eqref{gausslaw} with $\rho$ non-singular, then the amplitudes satisfy the following orthogonality conditions
\begin{equation}\label{eq:orthogonality for Ai}
A_1^i m_1^i+A_2^i m_2^i+A_3^i m_3^i=0,
\end{equation}
\begin{equation}\label{eq:orthogonality for Ar}
A_1^r m_1^i+A_2^r m_2^i+A_3^r m_3^r=0,
\end{equation}
and 
\begin{equation}\label{eq:orthogonality for At}
A_1^t m_1^i+A_2^t m_2^i+A_3^t m_3^t=0.
\end{equation}

\end{lemma}
\begin{proof}

From Theorem \ref{thm:MAIN THEOREM GSL}, \eqref{eq:generalized snell law with psi} holds and from the form of the fields \eqref{eq:Ei},
\eqref{eq:proposed form of the electric field}, \eqref{eq:proposed form of the electric field back}, and the notation \eqref{eq:notation for wave vectors} it follows that
\begin{align*}
\E_i(x,t)&=
{\bf A}_i e^{\i \(m_1^ix_1+m_2^i x_2+ m_3^ix_3-\omega\, t\)},\quad
\E_r(x,t)
=
{\bf A}_r e^{\i\(m_1^ix_1+m_2^i x_2+ m_3^rx_3-\omega\, t\)}\\
\E_t(x,t)
&={\bf A}_t e^{\i \(m_1^ix_1+m_2^i x_2+ m_3^tx_3-\omega\, t\)}.
\end{align*}

Since $\E$ is a distributional solution to \eqref{gausslaw}, it follows that $\E'$ satisfies \eqref{gausslaw} pointwise in $\Omega_-$ and $\E_t$ satisfies \eqref{gausslaw} pointwise in $\Omega_+$.
We then have for $x_3<0$ that
\begin{align*}
0&=\diver \, \E'=\diver \, \E_i+\diver \, \E_r\\
&=\i\,\(m_1^i,m_2^i,m_3^i\)\cdot {\bf A}_i \,e^{\i \(m_1^ix_1+m_2^i x_2+ m_3^ix_3-\omega\, t\)}
+
\i\,\(m_1^i,m_2^i,m_3^r\)\cdot {\bf A}_r \,e^{\i \(m_1^ix_1+m_2^i x_2+ m_3^rx_3-\omega\, t\)}\\
&=
\i\,e^{\i \(m_1^ix_1+m_2^i x_2-\omega\, t\)}
\(
\(m_1^i,m_2^i,m_3^i\)\cdot {\bf A}_i\,e^{\i m_3^ix_3}
+
\(m_1^i,m_2^i,m_3^r\)\cdot {\bf A}_r\,e^{\i m_3^rx_3}\)
\end{align*}
which implies
\[
\(m_1^i,m_2^i,m_3^i\)\cdot {\bf A}_i\,e^{\i \,m_3^ix_3}
+
\(m_1^i,m_2^i,m_3^r\)\cdot {\bf A}_r\,e^{\i \,m_3^rx_3}=0
\]
for all $x_3<0$.
Since $m_3^i>0$ and $m_3^r<0$, the exponentials in the last identity are linearly independent and 
therefore the coefficients must be zero, that is, \eqref{eq:orthogonality for Ai} and \eqref{eq:orthogonality for Ar} follow.

Since $\diver\,\E_t=0$ for $x_3>0$, \eqref{eq:orthogonality for At} also follows.

\end{proof}

\setcounter{equation}{0}
\section{Boundary conditions for the magnetic fields}\label{sec:boundary conditions for magnetic fields}

In this section, we derive the boundary conditions for the magnetic fields presented in Lemma \ref{eq:corresponding magnetic fields} from Theorem \ref{thm:general boundary conditions for Maxwell equations}. These boundary conditions will be utilized in Section \ref{sec:calculation of the amplitudes of the reflected and transmitted fields BIS} 
\cristian{in the calculation of the amplitudes.} 
 


\begin{proposition}\label{prop:boundary condition from M4} 
Under the assumptions of Theorem \ref{thm:MAIN THEOREM GSL}, if $\E$ and $\H$ are solutions to \eqref{amperemaxwelllaw}, with current density ${\bf J}$ satisfying the assumptions in Theorem \ref{thm:general boundary conditions for Maxwell equations} with $\nu_t=0$,  then 
\begin{align}\label{eq:relations with mupmA}
-(1/\mu_+)\(A_3^t m_1^i-m_3^t A_1^t\)
+(1/\mu_-)\(A_3^i m_1^i-m_3^i A_1^i +A_3^r m_1^i-m_3^rA_1^r\)
&=0\\
-(1/\mu_+)\(A_3^t m_2^i-m_3^t A_2^t \)
+(1/\mu_-) \(A_3^i m_2^i-m_3^i A_2^i
+A_3^r m_2^i-m_3^r A_2^r\)&=0\label{eq:relations with mupmB}.
\end{align}
\end{proposition}

\begin{proof}
From \eqref{eq:constitutive equations} we recall that 
$\D=\epsilon_-\,\E$ in $\Omega_-$, $\D=\epsilon_+\,\E$ in $\Omega_+$,
and  $\B=\mu_-\,\H$ in $\Omega_-$, $\B=\mu_+\,\H$ in $\Omega_+$.
Since $\J$ has not a singular part, then Theorem \ref{thm:general boundary conditions for Maxwell equations} Part (3) is applicable and we get that 
\begin{equation}\label{eq:boundary on H}
[[\H(X,t)]]\times \n(X)=0 
\end{equation}
for $X$ in the interface plane $\Gamma=\{x_3=0\}$.

From the expression of the electric fields in  \eqref{eq:Ei},  \eqref{eq:proposed form of the electric field}, \eqref{eq:proposed form of the electric field back}, and the corresponding magnetic fields obtained in Lemma \ref{eq:corresponding magnetic fields}, we have that for every $X=(x_1,x_2,0)\in \Gamma$
\begin{align*}
\lim_{x\to X,x\in \Omega_-}\H'(x,t)
&=
-\dfrac{c}{\mu_-}\({\bf A}_i\times \frac{\k_i}{v_1}\,\,e^{\i\,\omega\(\frac{\k_i\cdot X}{v_1}-t\)}
+{\bf A}_r\times \(\frac{\k_r}{v_1}+\nabla f(X)\)\,
e^{\i\,\omega\(\frac{\k_r\cdot X}{v_1}+f(X)-t\)}\)\\
\lim_{x\to X,x\in \Omega_+}\H_t(x,t)
&=
-\dfrac{c}{\mu_+}\,{\bf A}_t\times \(\dfrac{\k_t}{v_2}+\nabla f(X)\) \,
e^{\i\omega\(\frac{\k_t\cdot X}{v_2}+f(X)-t\)},
\end{align*}
where $\nabla f(X)=\(f_{x_1},f_{x_2},0\)$.
Substituting these into \eqref{eq:boundary on H}, we get
\begin{align*}\label{eq:boundary condition for the magnetic field H}
0&=-\dfrac{c}{\mu_+}\,\({\bf A}_t\times \(\dfrac{\k_t}{v_2}+\nabla f(X)\)\)\times \n \,
e^{\i\,\omega\(\frac{\k_t\cdot X}{v_2}+f(X)-t\)}\notag\\
&\qquad  +\dfrac{c}{\mu_-}\,\(\({\bf A}_i\times \frac{\k_i}{v_1}\)\times \n\,\,e^{\i\,\omega\(\frac{\k_i\cdot X}{v_1}-t\)}
+\({\bf A}_r\times \(\frac{\k_r}{v_1}+\nabla f(X)\)\)\times \n\,
e^{\i\,\omega\(\frac{\k_r\cdot X}{v_1}+f(X)-t\)}\).
\end{align*} 
Then from \eqref{eq:generalized snell law with psi}
 \begin{align*}
&-\dfrac{c}{\mu_+}\,\({\bf A}_t\times \(\dfrac{k_i^1}{v_1},\dfrac{k_i^2}{v_1},\dfrac{k_t^3}{v_2}\)\)\times \n \,
e^{\i\,\omega\(\frac{\k_r\cdot X}{v_1}+f(X)\)}\\
&\quad +\dfrac{c}{\mu_-}\,\(\({\bf A}_i\times \(\dfrac{k_i^1}{v_1},\dfrac{k_i^2}{v_1},\dfrac{k_i^3}{v_1}\)\)\times \n\,\,
e^{\i\,\omega\(\frac{\k_r\cdot X}{v_1}+\nabla f(X)\cdot X\)}
+\({\bf A}_r\times \(\dfrac{k_i^1}{v_1},\dfrac{k_i^2}{v_1},\dfrac{k_r^3}{v_2}\)\)\times \n\,
e^{\i\,\omega\(\frac{\k_r\cdot X}{v_1}+f(X)\)}\)
=0.
\end {align*}
From Theorem \ref{thm:MAIN THEOREM GSL}, $f(X)$ is affine, and since $f(0)=0$ we have $f(X)=\nabla f(X)\cdot X$. 
So canceling the exponential and the constant $c$ in the last equation we obtain
\begin{equation}\label{eq:another simplification}
-\dfrac{1}{\mu_+}\,\({\bf A}_t\times \(\dfrac{k_i^1}{v_1},\dfrac{k_i^2}{v_1},\dfrac{k_t^3}{v_2}\)\)\times \n 
 +\dfrac{1}{\mu_-}\,\(\({\bf A}_i\times \(\dfrac{k_i^1}{v_1},\dfrac{k_i^2}{v_1},\dfrac{k_i^3}{v_1}\)\)\times \n
+\({\bf A}_r\times \(\dfrac{k_i^1}{v_1},\dfrac{k_i^2}{v_1},\dfrac{k_r^3}{v_2}\)\)\times \n
\)
=0.
\end{equation}
%

Let us calculate the triple cross products, we use the formula $(a\times b)\times c=b(c\cdot a)-a(b\cdot c)$ and obtain
\begin{align*}
\({\bf A}_t\times \(\dfrac{k_i^1}{v_1},\dfrac{k_i^2}{v_1},\dfrac{k_t^3}{v_2}\)\)\times \n
&=
\({\bf A}_t\cdot \n \)\,\(\dfrac{k_i^1}{v_1},\dfrac{k_i^2}{v_1},\dfrac{k_t^3}{v_2}\)- \(\(\dfrac{k_i^1}{v_1},\dfrac{k_i^2}{v_1},\dfrac{k_t^3}{v_2}\)\cdot \n\)\,{\bf A}_t\\
&=\(A_3^t \frac{k_1^i}{v_1}-\frac{k_3^t}{v_2} A_1^t,A_3^t \frac{k_2^i}{v_1}-\frac{k_3^t}{v_2} A_2^t,0\)\\
\({\bf A}_r\times \(\dfrac{k_i^1}{v_1},\dfrac{k_i^2}{v_1},\dfrac{k_r^3}{v_2}\)\)\times \n&=\(A_3^r \frac{k_1^i}{v_1}-\frac{k_3^r}{v_1}A_1^r,A_3^r \frac{k_2^i}{v_1}-\frac{k_3^r}{v_1} A_2^r,0\)\\
\({\bf A}_i\times \(\dfrac{k_i^1}{v_1},\dfrac{k_i^2}{v_1},\dfrac{k_i^3}{v_1}\)\)\times \n
&=
\(A_3^i \frac{k_1^i}{v_1}-\frac{k_3^i}{v_1}A_1^i,A_3^i \frac{k_2^i}{v_1}-\frac{k_3^i}{v_1}A_2^i,0\).
\end{align*}
Replacing the obtained formulas above in \eqref{eq:another simplification}, and using the notation \eqref{eq:notation for wave vectors}, equations \eqref{eq:relations with mupmA} and \eqref{eq:relations with mupmB} follow.
\end{proof}

\begin{remark}\rm
We observe that analyzing the boundary condition \eqref{eq:density of mut bis bis} does not provide any additional information. Specifically, assume that the magnetic field ${\bf H}$ presented in Theorem \ref{thm:MAIN THEOREM GSL} is a distributional solution to \eqref{gausslawmagnetic}. 
 The jump on $\B$ is given by
\[
[[\B(X,t)]]=\lim_{x\to X,x\in \Omega_+}\mu_+ \H_t(x,t)-\lim_{x\to X,x\in\Omega_-}\mu_-\H'(x,t), \quad X\in \Gamma,
\]
and so from Lemma \ref{eq:corresponding magnetic fields} 
and \eqref{eq:density of mut bis bis}
\begin{align*}
[[\B(X,t)]]\cdot \n
&=-c\,\({\bf A}_t\times \(\dfrac{\k_t}{v_2}+\nabla f(X)\)\)\cdot \n \,
e^{\i\,\omega\(\frac{\k_t\cdot X}{v_2}+f(X)-t\)}\\
 &\qquad +
c\,\(\({\bf A}_i\times \frac{\k_i}{v_1}\)\cdot \n\,\,e^{\i\,\omega\(\frac{\k_i\cdot X}{v_1}-t\)}
+\({\bf A}_r\times \(\frac{\k_r}{v_1}+\nabla f(X)\)\)\cdot \n\,
e^{\i\,\omega\(\frac{\k_r\cdot X}{v_1}+f(X)-t\)}\)=0.
\end{align*}
As in the proof of Proposition \ref{prop:boundary condition from M4}, the exponentials are all equal 
and so cancelling them yields
\[
-\({\bf A}_t\times \(\dfrac{\k_t}{v_2}+\nabla f(X)\)\)\cdot \n
+
\({\bf A}_i\times \frac{\k_i}{v_1}\)\cdot \n +\({\bf A}_r\times \(\frac{\k_r}{v_1}+\nabla f(X)\)\)\cdot \n
=0.
\]
From the triple product formula $a\cdot (b\times c)=b\cdot (c\times a)$ and \eqref{eq:generalized snell law with psi}, it follows that
\begin{align*}
0&=
-{\bf A}_t\cdot \( \(\dfrac{k_1^i}{v_1},\dfrac{k_2^i}{v_1}, \dfrac{k_3^t}{v_2}\) \times \n \)
+
{\bf A}_i\cdot \( \(\frac{k_1^i}{v_1}, \frac{k_2^i}{v_1}, \frac{k_3^i}{v_1}\)\times  \n \)
+{\bf A}_r\cdot \(\(\frac{k_1^i}{v_1},\frac{k_2^i}{v_1},\frac{k_3^r}{v_1}\) \times \n\)\\
&=
-{\bf A}_t\cdot \(\frac{k_2^i}{v_1},-\frac{k_1^i}{v_1},0\)
+
{\bf A}_i\cdot \(\frac{k_2^i}{v_1},-\frac{k_1^i}{v_1},0\)
+
{\bf A}_r\cdot \(\frac{k_2^i}{v_1},-\frac{k_1^i}{v_1},0\)\\
&=
\(-{\bf A}_t+{\bf A}_i+{\bf A}_r\) \(\frac{k_2^i}{v_1},-\frac{k_1^i}{v_1},0\)
\end{align*}
which written in terms of $m$'s is 
\begin{equation}\label{eq:relation coming from the field B}
m_2^i\,\(-A_1^t+A_1^i+A_1^r\)-
m_1^i\,\(-A_2^t+A_2^i+A_2^r\)= 0.
\end{equation}
From equations \eqref{eq:equation second components} and \eqref{eq:equation first components}, equation \eqref{eq:relation coming from the field B} is satisfied. Consequently, the boundary condition \eqref{eq:density of mut bis bis} does not provide any additional equations for the amplitudes.
\end{remark}

\setcounter{equation}{0}
\section{Calculation of the amplitude coefficients}\label{sec:calculation of the amplitudes of the reflected and transmitted fields BIS}
Recall these amplitudes are 
$
{\bf A}_i=(A_1^i,A_2^i,A_3^i), {\bf A}_r=(A_1^r,A_2^r,A_3^r)$, and ${\bf A}_t=(A_1^t,A_2^t,A_3^t)$,
where the unknowns are ${\bf A}_r$ and ${\bf A}_t$. 
The purpose of the section is to find explicit formulas for ${\bf A}_r$ and ${\bf A}_t$ in terms of ${\bf A}_i$ and the wave vectors.
\cristian{Listing all the relationships obtained for the amplitudes in a table we get}
\[
\begin{matrix}
A_1^r & A_2^r & A_3^r & A_1^t & A_2^t & A_3^t &  A_1^i & A_2^i & A_3^i & \text{From equation}\\
\hline
1 & 0 & 0 & -1 & 0 & 0 & 1 & 0 & 0 & \eqref{eq:equation first components}\\
0 & -1 & 0 & 0 & 1 & 0 & 0 & -1 & 0 & \eqref{eq:equation second components}\\
0 & 0 & \epsilon_- & 0 & 0 & -\epsilon_+ & 0 & 0 & \epsilon_- & \eqref{eq:equation third components}\\
-m_3^r/\mu_- & 0 & m_1^i/\mu_- & m_3^t/\mu_+ & 0 & -m_1^i/\mu_+ & -m_3^i/\mu_-& 0 & m_1^i/\mu_- & \eqref{eq:relations with mupmA}\\
0 & -m_3^r/\mu_- & m_2^i/\mu_- & 0 & m_3^t/\mu_+ & -m_2^i/\mu_+ & 0 & -m_3^i/\mu_- & m_2^i/\mu_- &  \eqref{eq:relations with mupmB}\\
0 & 0 & 0 & 0 & 0 & 0 & m_1^i & m_2^i & m_3^i & \eqref{eq:orthogonality for Ai}\\
m_1^i & m_2^i & m_3^r  & 0 & 0 & 0 & 0 & 0 & 0 & \eqref{eq:orthogonality for Ar}\\
 0 & 0 & 0 & m_1^i & m_2^i & m_3^t  & 0 & 0 & 0 & \eqref{eq:orthogonality for At}
 \end{matrix}.
\]
\cristian{We remark that the first five rows in the table, \eqref{eq:equation first components}, \eqref{eq:equation second components},
\eqref{eq:equation third components}, \eqref{eq:relations with mupmA}, and \eqref{eq:relations with mupmB} are all consequences of the boundary conditions, and the last three rows in the table, \eqref{eq:orthogonality for Ai}, \eqref{eq:orthogonality for Ar}, and \eqref{eq:orthogonality for At}, are orthogonality type conditions derived in Lemma \ref{lm:orthogonality conditions for amplitudes} from the fact that $\E,\H$ satisfy the Maxwell system. Notice that we have six unknowns, the components of ${\bf A}_t$ and ${\bf A}_r$, and eight equations. The system will be first reduced to solve \eqref{eq:equation for At in terms of Ai} for ${\bf A}_t$, and substitute this value in \eqref{eq:Ar in terms of At and Ai} to get ${\bf A}_r$. A necessary and sufficient condition for the solvability of the system \eqref{eq:equation for At in terms of Ai} is given in the following section.}
The coefficient matrix of the system is then
\[
M=\begin{bmatrix}
1 & 0 & 0 & -1 & 0 & 0 & 1 & 0 & 0 \\
0 & -1 & 0 & 0 & 1 & 0 & 0 & -1 & 0 \\
0 & 0 & \epsilon_- & 0 & 0 & -\epsilon_+ & 0 & 0 & \epsilon_- \\
-m_3^r/\mu_- & 0 & m_1^i/\mu_- & m_3^t/\mu_+ & 0 & -m_1^i/\mu_+ & -m_3^i/\mu_-& 0 & m_1^i/\mu_- \\
0 & -m_3^r/\mu_- & m_2^i/\mu_- & 0 & m_3^t/\mu_+ & -m_2^i/\mu_+ & 0 & -m_3^i/\mu_- & m_2^i/\mu_- \\
0 & 0 & 0 & 0 & 0 & 0 & m_1^i & m_2^i & m_3^i \\
m_1^i & m_2^i & m_3^r  & 0 & 0 & 0 & 0 & 0 & 0 \\
 0 & 0 & 0 & m_1^i & m_2^i & m_3^t  & 0 & 0 & 0 
 \end{bmatrix}.
\]
For the calculation of the amplitudes it is convenient to write this matrix in terms of blocks. 
If we set
\begin{align*}
M_1
&=
\begin{bmatrix}
1 & 0 & 0 \\
0 & -1 & 0  \\
0 & 0 & \epsilon_- 
 \end{bmatrix},\qquad
 M_2
 =
\begin{bmatrix}
-1 & 0 & 0  \\
 0 & 1 & 0  \\
 0 & 0 & -\epsilon_+ 
 \end{bmatrix} \\
N_1
&=
\begin{bmatrix}
-m_3^r/\mu_- & 0 & m_1^i/\mu_-  \\
0 & -m_3^r/\mu_- & m_2^i/\mu_- \\
0 & 0 & 0  \\
m_1^i & m_2^i & m_3^r   \\
 0 & 0 & 0  
 \end{bmatrix},\qquad
 N_2
 =
 \begin{bmatrix}
 m_3^t/\mu_+ & 0 & -m_1^i/\mu_+  \\
 0 & m_3^t/\mu_+ & -m_2^i/\mu_+ \\
0 & 0 & 0  \\
 0 & 0 & 0  \\
  m_1^i & m_2^i & m_3^t   
 \end{bmatrix},\\
 N_3
 &=
 \begin{bmatrix}
 -m_3^i/\mu_-& 0 & m_1^i/\mu_- \\
 0 & -m_3^i/\mu_- & m_2^i/\mu_- \\
 m_1^i & m_2^i & m_3^i \\
 0 & 0 & 0 \\
  0 & 0 & 0 
 \end{bmatrix},
\end{align*}
then we can write 
$
M
=
\begin{bmatrix}
M_1 & M_2 & M_1\\
N_1 & N_2 & N_3
\end{bmatrix}$.
Therefore the amplitudes must verify the equations, \cristian{written with ${\bf A}_r,{\bf A}_t,{\bf A}_i$ column 3-vectors},
\[
\begin{bmatrix}
M_1 & M_2 & M_1\\
N_1 & N_2 & N_3
\end{bmatrix}
\begin{bmatrix}
{\bf A}_r\\
{\bf A}_t\\
{\bf A}_i
\end{bmatrix}
=
0\in \R^8
\]
which means
\begin{equation*}
M_1 {\bf A}_r+M_2 {\bf A}_t+M_1{\bf A}_i
=0\in \R^3,\qquad 
N_1 {\bf A}_r+N_2 {\bf A}_t+N_3{\bf A}_i
=0\in \R^5.
\end{equation*}
From the first equation and since the matrix $M_1$ is invertible
\begin{equation}\label{eq:Ar in terms of At and Ai}
{\bf A}_r=-M_1^{-1}M_2 {\bf A}_t -{\bf A}_i
\end{equation}
which substituted in the second equation yields
\begin{equation}\label{eq:equation for At in terms of Ai}
\(N_2 -N_1 M_1^{-1}M_2 \){\bf A}_t=\(N_1-N_3\){\bf A}_i.
\end{equation}
Now
\[
P =
N_2 -N_1 M_1^{-1}M_2= 
\begin{bmatrix}
-\dfrac{m_3^r}{\mu_-} + \dfrac{m_3^t}{\mu_+} & 0 & \(\dfrac{\epsilon_+ }{\epsilon_- \mu_-} - \dfrac{1}{\mu_+}\)m_1^i \\
0 & -\dfrac{m_3^r}{\mu_-} + \dfrac{m_3^t}{\mu_+} & \(\dfrac{\epsilon_+ }{\epsilon_- \mu_-} - \dfrac{1}{\mu_+}\)m_2^i \\
0 & 0 & 0 \\
m_1^i & m_2^i & \dfrac{\epsilon_+ m_3^r}{\epsilon_-} \\
m_1^i & m_2^i & m_3^t
\end{bmatrix}
\]
and
\[
Q =N_1-N_3
=\begin{bmatrix}
\dfrac{m_3^i - m_3^r}{\mu_-} & 0 & 0 \\
0 & \dfrac{m_3^i - m_3^r}{\mu_-} & 0 \\
- m_1^i & - m_2^i & - m_3^i \\
m_1^i & m_2^i & m_3^r \\
0 & 0 & 0
\end{bmatrix}.
\]
Next 
\[
Q {\bf A}_i =
\begin{bmatrix}
\displaystyle \frac{A_1^i (m_3^i - m_3^r)}{\mu_-} \\
\displaystyle \frac{A_2^i (m_3^i - m_3^r)}{\mu_-} \\
- A_1^i m_1^i - A_2^i m_2^i - A_3^i m_3^i \\
A_1^i m_1^i + A_2^i m_2^i + A_3^i m_3^r \\
0
\end{bmatrix}
=
\begin{bmatrix}
\displaystyle \frac{A_1^i (m_3^i - m_3^r)}{\mu_-} \\
\displaystyle \frac{A_2^i (m_3^i - m_3^r)}{\mu_-} \\
0 \\
A_3^i \(m_3^r-m_3^i\) \\
0
\end{bmatrix}.
\]
Therefore, \eqref{eq:equation for At in terms of Ai} has a solution ${\bf A}_t$ if the vector $QA_i$ is in the column space of the matrix $P$, and hence ${\bf A}_r$ follows from \eqref{eq:Ar in terms of At and Ai}. 

\subsection{Solvability of the system \eqref{eq:equation for At in terms of Ai}}

We shall prove the following proposition.
\begin{proposition}\label{prop:formulas for amplitudes}
The system \eqref{eq:equation for At in terms of Ai} is solvable if and only if 
\begin{equation}\label{eq:solvability condition involving A3i}
A_3^i m_3^i\dfrac{1}{\mu_-}
=
\(-\delta  \,m_3^t +\alpha_1 m_1^i + \alpha_2 m_2^i\)A_3^i\dfrac{1}{\alpha_3-m_3^t}
\end{equation} 
holds, where 
\begin{align*}
\delta = \frac{m_3^t}{\mu_+} - \frac{m_3^r}{\mu_-}, \quad
\alpha_1 = 
\(\frac{\epsilon_+ }{\epsilon_- \mu_-} - \frac{1}{\mu_+}\)m_1^i, 
\quad \alpha_2 &
=
\(\frac{\epsilon_+ }{\epsilon_- \mu_-} - \frac{1}{\mu_+}\)m_2^i, \quad
\alpha_3 = \frac{\epsilon_+ m_3^r}{\epsilon_-};
\end{align*}
notice that $\delta>0$ and $\alpha_3<0$ since $m_3^t>0$ and $m_3^r<0$.
Moreover, we obtain the following formula for the amplitude ${\bf A}_t$:
\begin{equation}\label{eq:formulas for At in general}
\begin{cases}
A_1^t
&=
\dfrac{m_3^i - m_3^r}{-\dfrac{m_3^r}{\mu_-} + \dfrac{m_3^t}{\mu_+}}
\(\dfrac{A_1^i }{\mu_-}
-
 \(\dfrac{\epsilon_+ }{\epsilon_- \mu_-} - \dfrac{1}{\mu_+}\)m_1^i
\dfrac{ \epsilon_-}{\epsilon_-m_3^t-\epsilon_+ m_3^r}A_3^i\)\\
A_2^t
&=
\dfrac{m_3^i - m_3^r}{-\dfrac{m_3^r}{\mu_-} + \dfrac{m_3^t}{\mu_+}}
\(\dfrac{A_2^i }{\mu_-}
-
 \(\dfrac{\epsilon_+ }{\epsilon_- \mu_-} - \dfrac{1}{\mu_+}\)m_2^i
\dfrac{ \epsilon_-}{\epsilon_-m_3^t-\epsilon_+ m_3^r}A_3^i\)\\
A_3^t&=\dfrac{ \epsilon_-\(m_3^i-m_3^r\)}{\epsilon_- m_3^t-\epsilon_+ m_3^r}A_3^i.
\end{cases}
\end{equation}
Inserting this value of ${\bf A}_t$ into \eqref{eq:Ar in terms of At and Ai} we obtain the amplitude ${\bf A}_r$.
\end{proposition}

\begin{proof}
\cristian{Since the third row of the system \eqref{eq:equation for At in terms of Ai} is zero, the system of equations to solve is}
\[
\begin{bmatrix}
-\dfrac{m_3^r}{\mu_-} + \dfrac{m_3^t}{\mu_+} & 0 & \(\dfrac{\epsilon_+ }{\epsilon_- \mu_-} - \dfrac{1}{\mu_+}\)m_1^i \\
0 & -\dfrac{m_3^r}{\mu_-} + \dfrac{m_3^t}{\mu_+} & \(\dfrac{\epsilon_+ }{\epsilon_- \mu_-} - \dfrac{1}{\mu_+}\)m_2^i \\
m_1^i & m_2^i & \dfrac{\epsilon_+ m_3^r}{\epsilon_-} \\
m_1^i & m_2^i & m_3^t
\end{bmatrix}
\begin{bmatrix}
A_1^t\\
A_2^t\\
A_3^t
\end{bmatrix}
=
\begin{bmatrix}
\displaystyle \frac{A_1^i (m_3^i - m_3^r)}{\mu_-} \\
\displaystyle \frac{A_2^i (m_3^i - m_3^r)}{\mu_-} \\
A_3^i \(m_3^r-m_3^i\) \\
0
\end{bmatrix}.
\]
Subtracting the third equation from the fourth we need to solve
\[
\begin{bmatrix}
-\dfrac{m_3^r}{\mu_-} + \dfrac{m_3^t}{\mu_+} & 0 & \(\dfrac{\epsilon_+ }{\epsilon_- \mu_-} - \dfrac{1}{\mu_+}\)m_1^i \\
0 & -\dfrac{m_3^r}{\mu_-} + \dfrac{m_3^t}{\mu_+} & \(\dfrac{\epsilon_+ }{\epsilon_- \mu_-} - \dfrac{1}{\mu_+}\)m_2^i \\
m_1^i & m_2^i & \dfrac{\epsilon_+ m_3^r}{\epsilon_-} \\
0 & 0 & m_3^t-\dfrac{\epsilon_+ m_3^r}{\epsilon_-}
\end{bmatrix}
\begin{bmatrix}
A_1^t\\
A_2^t\\
A_3^t
\end{bmatrix}
=
\begin{bmatrix}
\displaystyle \frac{A_1^i (m_3^i - m_3^r)}{\mu_-} \\
\displaystyle \frac{A_2^i (m_3^i - m_3^r)}{\mu_-} \\
A_3^i \(m_3^r-m_3^i\) \\
-A_3^i \(m_3^r-m_3^i\)
\end{bmatrix},
\]
which written in terms of $\delta,\alpha_1,\alpha_2,\alpha_3$ is the system
\[
\begin{bmatrix}
\delta & 0 & \alpha_1 \\
0 & \delta & \alpha_2 \\
m_1^i & m_2^i & \alpha_3 \\
0 & 0 & m_3^t-\alpha_3
\end{bmatrix}
\begin{bmatrix}
A_1^t\\
A_2^t\\
A_3^t
\end{bmatrix}
=
\begin{bmatrix}
\displaystyle \frac{A_1^i (m_3^i - m_3^r)}{\mu_-} \\
\displaystyle \frac{A_2^i (m_3^i - m_3^r)}{\mu_-} \\
A_3^i \(m_3^r-m_3^i\) \\
-A_3^i \(m_3^r-m_3^i\)
\end{bmatrix}.
\]
Dividing rows one and two by $\delta$ gives
\[
\begin{bmatrix}
1 & 0 & \alpha_1/\delta \\
0 & 1 & \alpha_2/\delta \\
m_1^i & m_2^i & \alpha_3 \\
0 & 0 & m_3^t-\alpha_3
\end{bmatrix}
\begin{bmatrix}
A_1^t\\
A_2^t\\
A_3^t
\end{bmatrix}
=
\begin{bmatrix}
\displaystyle \frac{A_1^i (m_3^i - m_3^r)}{\delta \mu_-} \\
\displaystyle \frac{A_2^i (m_3^i - m_3^r)}{\delta \mu_-} \\
A_3^i \(m_3^r-m_3^i\) \\
-A_3^i \(m_3^r-m_3^i\)
\end{bmatrix};
\]
multiplying the first row by $-m_1^i$ and add in it to the third row gives
\[
\begin{bmatrix}
1 & 0 & \alpha_1/\delta \\
0 & 1 & \alpha_2/\delta \\
0 & m_2^i & \alpha_3-m_1^i (\alpha_1/\delta) \\
0 & 0 & m_3^t-\alpha_3
\end{bmatrix}
\begin{bmatrix}
A_1^t\\
A_2^t\\
A_3^t
\end{bmatrix}
=
\begin{bmatrix}
\displaystyle \frac{A_1^i (m_3^i - m_3^r)}{\delta \mu_-} \\
\displaystyle \frac{A_2^i (m_3^i - m_3^r)}{\delta \mu_-} \\
A_3^i \(m_3^r-m_3^i\)-m_1^i \dfrac{A_1^i (m_3^i - m_3^r)}{\delta \mu_-} \\
-A_3^i \(m_3^r-m_3^i\)
\end{bmatrix};
\]
multiplying the second row by $-m_2^i$ and add in it to the third row gives
\[
\begin{bmatrix}
1 & 0 & \alpha_1/\delta \\
0 & 1 & \alpha_2/\delta \\
0 & 0 & \alpha_3-m_1^i (\alpha_1/\delta)-m_2^i (\alpha_2/\delta) \\
0 & 0 & m_3^t-\alpha_3
\end{bmatrix}
\begin{bmatrix}
A_1^t\\
A_2^t\\
A_3^t
\end{bmatrix}
=
\begin{bmatrix}
\displaystyle \frac{A_1^i (m_3^i - m_3^r)}{\delta \mu_-} \\
\displaystyle \frac{A_2^i (m_3^i - m_3^r)}{\delta \mu_-} \\
A_3^i \(m_3^r-m_3^i\)-m_1^i \dfrac{A_1^i (m_3^i - m_3^r)}{\delta \mu_-}-m_2^i \dfrac{A_2^i (m_3^i - m_3^r)}{\delta \mu_-} \\
-A_3^i \(m_3^r-m_3^i\)
\end{bmatrix}.
\]
From the fourth equation, the value of $A_3^t$ is given by
\[
A_3^t=\dfrac{-A_3^i \(m_3^r-m_3^i\)}{m_3^t-\alpha_3}.
\]
We shall verify that \eqref{eq:solvability condition involving A3i} implies that this value of $A_3^t$ satisfies the third equation, i.e., under \eqref{eq:solvability condition involving A3i} equations three and four are equivalent. In fact, substituting this value of $A_3^t$ on the left hand side of the third equation we need to prove the identity
\[
\(\alpha_3-m_1^i (\alpha_1/\delta)-m_2^i (\alpha_2/\delta)\)\dfrac{A_3^i \(m_3^i-m_3^r\)}{m_3^t-\alpha_3}
=
A_3^i \(m_3^r-m_3^i\)-m_1^i \dfrac{A_1^i (m_3^i - m_3^r)}{\delta \mu_-}-m_2^i \dfrac{A_2^i (m_3^i - m_3^r)}{\delta \mu_-}.\]
Since $m_1^iA_1^i+m_2^i A_2^i=-m_3^i A_3^i$, this identity is equivalent to
\[
\(\alpha_3-m_1^i (\alpha_1/\delta)-m_2^i (\alpha_2/\delta)\)\dfrac{A_3^i \(m_3^i-m_3^r\)}{m_3^t-\alpha_3}
=
A_3^i \(m_3^r-m_3^i\)+m_3^i A_3^i\dfrac{m_3^i - m_3^r}{\delta \mu_-}.\]
Now notice that moving terms around and since $m_3^i-m_3^r\neq 0$, the last identity is equivalent to \eqref{eq:solvability condition involving A3i}.
Hence \eqref{eq:formulas for At in general} follows.

\cristian{Notice that given ${\bf A}_i$, the amplitude solution ${\bf A}_t$, and consequently ${\bf A}_r$, are unique since the null space of the matrix 
\[
\begin{bmatrix}
1 & 0 & \alpha_1/\delta \\
0 & 1 & \alpha_2/\delta \\
0 & 0 & \alpha_3-m_1^i (\alpha_1/\delta)-m_2^i (\alpha_2/\delta) \\
0 & 0 & m_3^t-\alpha_3
\end{bmatrix}
\]
is zero since $m_3^r<0$ and so $m_3^t-\alpha_3>0$.} 

\end{proof}

\subsection{Analysis of the condition \eqref{eq:solvability condition involving A3i}}

In case $A_3^i= 0$ (the incident wave is transverse electric (TE), that is, it is perpendicular to the normal to the interface), condition \eqref{eq:solvability condition involving A3i} obviously holds
and from \eqref{eq:formulas for At in general} the amplitudes are 
\[
\(A_1^t,A_2^t,A_3^t\)= \dfrac{\mu_+\(m_3^i-m_3^r\)}{\mu_- m_3^t-\mu_+m_3^r}\(A_1^i,A_2^i,0\),
\]
and hence from \eqref{eq:Ar in terms of At and Ai}
\[
\(A_1^r,A_2^r,A_3^r\)=
\dfrac{\mu_+\(m_3^i-m_3^r\)}{\mu_- m_3^t-\mu_+m_3^r}
\begin{bmatrix}
1 & 0 & 0\\
0 & 1 & 0\\
0 & 0 & \epsilon_+/\epsilon_-
\end{bmatrix}
{\bf A}_i
-{\bf A}_i
=
\(\dfrac{\mu_+\(m_3^i-m_3^r\)}{\mu_- m_3^t-\mu_+m_3^r}-1\)
\(A_1^i,A_2^i,0\).
\]

In case $A_3^i\neq  0$,
if \eqref{eq:solvability condition involving A3i} holds, then some relationships between the wave vectors $m^\ell$, $\ell=i,r,t$, must be satisfied.
Indeed, cancelling $A_3^i$ in \eqref{eq:solvability condition involving A3i} for the solvability of the system we must have
\begin{align*}
m_3^i\dfrac{\frac{\epsilon_+ m_3^r}{\epsilon_-}-m_3^t}{\mu_-}=m_3^i\dfrac{\alpha_3-m_3^t}{\mu_-}
&=
-\delta  \,m_3^t +\alpha_1 m_1^i + \alpha_2 m_2^i\\
&=
-\(\frac{m_3^t}{\mu_+} - \frac{m_3^r}{\mu_-}\)  \,m_3^t 
+\(\frac{\epsilon_+ }{\epsilon_- \mu_-} - \frac{1}{\mu_+}\)\(\(m_1^i\)^2 
+\(m_2^i\)^2\).
\end{align*}
Re writing this expression we get 
\begin{equation}\label{eq:condition for solvability 5 by 3 matrix}
m_3^i\dfrac{\epsilon_+ m_3^r- \epsilon_- m_3^t}{\mu_- \epsilon_-}
=
-\(\frac{m_3^t}{\mu_+} - \frac{m_3^r}{\mu_-}\)  \,m_3^t 
+\(\frac{\epsilon_+ \mu_+-\epsilon_- \mu_-}{\epsilon_- \mu_- \mu_+} \)
\(\(m_1^i\)^2 +\(m_2^i\)^2\).
\end{equation}
Since $(m_1^i)^2+(m_2^i)^2+(m_3^i)^2=\omega^2/v_1^2=\omega^2 \,\epsilon_- \mu_-$, 
we have $(m_1^i)^2+(m_2^i)^2=\omega^2 \,\epsilon_- \mu_- -(m_3^i)^2$ which substituted in the last expression yields
\[
m_3^i\dfrac{\epsilon_+ m_3^r- \epsilon_- m_3^t}{\mu_- \epsilon_-}
=
-\(\frac{m_3^t}{\mu_+} - \frac{m_3^r}{\mu_-}\)  \,m_3^t 
+\(\frac{\epsilon_+ \mu_+-\epsilon_- \mu_-}{\epsilon_- \mu_- \mu_+} \)
\(\omega^2 \,\epsilon_- \mu_- -(m_3^i)^2\).
\]
From \eqref{eq:formula for m3t in terms of m3r},
$
m_3^t = \sqrt{ \omega^2(\epsilon_+ \mu_+ - \epsilon_- \mu_-) + (m_3^r)^2 }$,
which substituted in the last expression we get after simplification that the difference between the left-hand and right-hand sides becomes:
\[
(m_3^i + m_3^r)
\left(
\epsilon_+ \mu_+ m_3^i  - \epsilon_- \mu_- m_3^i  + \epsilon_- \mu_- m_3^r  - \epsilon_- \mu_+ 
m_3^t
\right)
=0
\]
Thus, \eqref{eq:condition for solvability 5 by 3 matrix} holds if and only if:
\begin{equation}\label{eq:first compatibility condition}
m_3^i + m_3^r = 0
\end{equation}
or
\begin{equation}\label{eq:second compatibility condition}
\(\epsilon_+ \mu_+ - \epsilon_- \mu_- \)m_3^i
= \epsilon_- \(\mu_+ m_3^t-\mu_- m_3^r\).
\end{equation}
From Equation \eqref{eq:m3r}
$
m_3^r=-\sqrt{\omega^2 \mu_- \epsilon_- -\(m_1^i-\omega f_{x_1}\)^2-\(m_2^i-\omega f_{x_2}\)^2}$,
and since $m_3^i=\sqrt{\omega^2 \mu_- \epsilon_--\(m_1^i\)^2-\(m_2^i\)^2}$, \eqref{eq:first compatibility condition} means
\[
\sqrt{\omega^2 \mu_- \epsilon_- -\(m_1^i-\omega f_{x_1}\)^2-\(m_2^i-\omega f_{x_2}\)^2}
=
\sqrt{\omega^2 \mu_- \epsilon_- -\(m_1^i\)^2-\(m_2^i\)^2}
\]
which implies $\(m_1^i-\omega f_{x_1}\)^2+\(m_2^i-\omega f_{x_2}\)^2=\(m_1^i\)^2+\(m_2^i\)^2$.
So \eqref{eq:first compatibility condition} holds if and only if the gradient of $f$ satisfies
\[
\omega \((f_{x_1})^2+(f_{x_2})^2\)=2\(m_1^i f_{x_1}+m_2^i f_{x_2}\).
\]
This is clearly satisfied for the standard refraction case when $f=0$.

Suppose $m_3^i + m_3^r \neq  0$. Notice that since $m_3^i>0, m_3^r<0$, and $m_3^t>0$, it follows that  
$\mu_+ m_3^t-\mu_- m_3^r>0$
and so $\epsilon_+ \mu_+ > \epsilon_- \mu_-$.
Therefore if $\epsilon_+ \mu_+ \leq \epsilon_- \mu_-$, then from \eqref{eq:second compatibility condition} for the system \eqref{eq:equation for At in terms of Ai} to be solvable we must have \eqref{eq:first compatibility condition} or $A_3^i=0$.

\section{Appendix}\label{app:proof of exponential lemma}
Here we prove Lemma \ref{lm:exponential gradient} used in the proof of Theorem \ref{thm:MAIN THEOREM GSL}.

%
\begin{proof}
Differentiating \eqref{eq:general exponential equation a b c} with respect to $x_1$ and dividing by $\i$ yields
\begin{equation}\label{eq:differentiated with respect to x1 equation 1 BIS}
a \,m_1^i\,e^{\i\,\(\m_i\cdot X\)}
+b\,\(m_1^r+\psi_{x_1}\) \,e^{\i\,\(\m_r\cdot X+\psi(X)\)}
+
c \,\(m_1^t+\psi_{x_1}\)\,e^{\i\,\(\m_t\cdot X+\psi(X)\)}=0.
\end{equation}
Next differentiate \eqref{eq:differentiated with respect to x1 equation 1 BIS} with respect to $x_1$ to get
\begin{align*}
&\i\,a \,(m_1^i)^2\,e^{\i\,\(\m_i\cdot X\)}
+b\,\(\psi_{x_1x_1} 
+\i\,\(m_1^r+\psi_{x_1}\)^2\) \,e^{\i\,\(\m_r\cdot X+\psi(X)\)}\\
&\qquad  +
c \,\(\psi_{x_1x_1}
+
\i\,\(m_1^t+\psi_{x_1}\)^2\)\,e^{\i\,\(\m_t\cdot X+\psi(X)\)}
=0.
\end{align*}
Putting together \eqref{eq:general exponential equation a b c}, \eqref{eq:differentiated with respect to x1 equation 1 BIS}, and the last equation yields the following system
\[
\(
\begin{matrix}
1 & 1 & 1\\
m_1^i & m_1^r+\psi_{x_1}  &  m_1^t+\psi_{x_1} \\
\i\,(m_1^i)^2 &
 \psi_{x_1x_1} 
+\i\,\(m_1^r+\psi_{x_1}\)^2  &
 \psi_{x_1x_1}
+
\i\,\(m_1^t+\psi_{x_1}\)^2 \\
\end{matrix}
\)
\(
\begin{matrix}
a\,e^{\i\,\(\m_i\cdot X\)}\\ 
b\,e^{\i\,\(\m_r\cdot X+\psi(X)\)}\\ 
c\,e^{\i\,\(\m_t\cdot X+\psi(X)\)}\\
\end{matrix}
\)=0.
\]
If $a\neq 0$, or $b\neq 0$,  or $c\neq 0$, then the vector $\(
\begin{matrix}
a\,e^{\i\,\(\m_i\cdot X\)}\\ 
b\,e^{\i\,\(\m_r\cdot X+\psi(X)\)}\\ 
c\,e^{\i\,\(\m_t\cdot X+\psi(X)\)}\\
\end{matrix}
\)$ is a non trivial solution to the system and therefore the matrix 
\begin{align*}
M(x_1,x_2)= \(
\begin{matrix}
1 & 1 & 1\\
m_1^i & (m_1^r+\psi_{x_1} & m_1^t+\psi_{x_1}\\
\i\,(m_1^i)^2 &
\psi_{x_1x_1} 
+\i\,\(m_1^r+\psi_{x_1}\)^2 &
\psi_{x_1x_1}
+
\i\,\(m_1^t+\psi_{x_1}\)^2\\
\end{matrix}
\)
\end{align*}
has determinant equals zero.
Multiplying the first row of $M$ by $-m_1^i$ and adding it to the second row, and 
multiplying the first row by $-\i\,(m_1^i)^2$ and adding it to the third row yields
\[
\(
\begin{matrix}
1 & 1 & 1\\
0 & -m_1^i+m_1^r+\psi_{x_1} & -m_1^i+m_1^t+\psi_{x_1}\\
0 &
\psi_{x_1x_1} 
+\i\,\(\(m_1^r+\psi_{x_1}\)^2-(m_1^i)^2\) &
\psi_{x_1x_1}
+
\i\,\(\(m_1^t+\psi_{x_1}\)^2-(m_1^i)^2\)\\
\end{matrix}
\).
\]
So the determinant of the last matrix is zero and 
factoring the difference of squares yields
\begin{align*}
&\(-m_1^i+m_1^r+\psi_{x_1}\)
\(\psi_{x_1x_1}+\i\,\(m_1^t+\psi_{x_1}-m_1^i\)\(m_1^t+\psi_{x_1}+m_1^i\)
\)\\
&\qquad -
\(-m_1^i+m_1^t+\psi_{x_1}\)
\(\psi_{x_1x_1} 
+\i\,\(m_1^r+\psi_{x_1}-m_1^i\)\(m_1^i+m_1^r+\psi_{x_1}\)\)=0
\end{align*}
that is,
\begin{equation}\label{eq: equation equals zero for first components BIS}
\(m_1^r-m_1^t\)\,
\(\psi_{x_1x_1}-\i\,\(-m_1^i+m_1^r+\psi_{x_1}\)\,\(m_1^t+\psi_{x_1}-m_1^i\)\)=0.
\end{equation}
Proceeding in the same way differentiating \eqref{eq:general exponential equation a b c}
with respect to $x_2$ we obtain for the second components the equation
\begin{equation}\label{eq: equation equals zero for second components BIS}
\(m_2^r-m_2^t\)\,
\(\psi_{x_2x_2}-\i\,\(-m_2^i+m_2^r+\psi_{x_2}\)\,\(m_2^t+\psi_{x_2}-m_2^i\)\)=0.
\end{equation}

{\it Proof of (i).}
We shall prove first that $m_1^t= m_1^r$, and a similar argument proves that $m_2^t= m_2^r$. 
Suppose by contradiction that $m_1^t\neq m_1^r$, then from \eqref{eq: equation equals zero for first components BIS} we must have
\[
\psi_{x_1x_1}-\i\,\(-m_1^i+m_1^r+\psi_{x_1}\)\,\(m_1^t+\psi_{x_1}-m_1^i\)=0
\]
and since the wave vectors $m^\ell$ have real components for $\ell=i,r,t$ and the function $\psi$ is real valued we obtain that
\[
\psi_{x_1x_1}(x_1,x_2)=0 \text{ and } 
\(-m_1^i+m_1^r+\psi_{x_1}\)\,\(m_1^t+\psi_{x_1}-m_1^i\)=0.
\]
which implies that 
$
\psi_{x_1}(x_1,x_2)=g(x_2)
$,
and 
\begin{equation}\label{eq:pair of equations for mr and mt BIS}
m_1^i-m_1^r=\psi_{x_1}(x_1,x_2),\text{  or  } m_1^i-m_1^t=\psi_{x_1}(x_1,x_2)
\end{equation}
for all $x_1,x_2$. Since $\psi_{x_1}$ is continuous we get that $\psi_{x_1}$ is constant. 
Then $\psi(x_1,x_2)=c_0\,x_1+h(x_2)$ with $c_0$ constant.
Also at least one of the equations in \eqref{eq:pair of equations for mr and mt BIS} holds, so $m_1^i-m_1^r=c_0$ or $m_1^i-m_1^t=c_0$. Suppose first that $m_1^i-m_1^r=c_0$.
Substituting the value of $m_1^r$ into \eqref{eq:general exponential equation a b c} and letting $x_2=0$ yields
\begin{align*}
&a \,e^{\i\,\(m_1^i\,x_1\)}
+b \,e^{\i\,\(m_1^r\,x_1+\psi(x_1,0)\)}
+
c \,e^{\i\,\(m_1^t\,x_1+\psi(x_1,0)\)}\\
&=a \,e^{\i\,\(m_1^i\,x_1\)}
+b \,e^{\i\,\(m_1^r\,x_1+c_0\,x_1+h(0)\)}
+
c \,e^{\i\,\(m_1^t\,x_1+c_0\,x_1+h(0)\)}\\
&=
a \,e^{\i\,\(m_1^i\,x_1\)}
+b \,e^{\i\,\(m_1^i\,x_1+h(0)\)}
+
c \,e^{\i\,\(m_1^t\,x_1+c_0\,x_1+h(0)\)}\\
&=
\(a \,
+b \,e^{\i\,h(0)}\)\,e^{\i\,\(m_1^i\,x_1\)}
+
c \,e^{\i\,\(m_1^t\,x_1+c_0\,x_1+h(0)\)}=0.
\end{align*}
Letting 
$
A=a \,
+b \,e^{\i\,h(0)}
$
and  
$B=c\,e^{\i\,h(0)}$
we obtain that
\[
A\,e^{\i\,\(m_1^i\,x_1\)}
+
B \,e^{\i\,\(m_1^t\,x_1+c_0\,x_1\)}=0
\]
for all $x_1$. 
Differentiating the last expression with respect to $x_1$ and dividing by $\i$ yields
\[
A\,m_1^i\,e^{\i\,\(m_1^i\,x_1\)}
+
B \(m_1^t+c_0\)\,e^{\i\,\(m_1^t\,x_1+c_0\,x_1\)}=0
\]
which written in matrix form is
\[
\(
\begin{matrix}
1 & 1\\
m_1^i & m_1^t+c_0
\end{matrix}
\)
\(
\begin{matrix}
A\,e^{\i\,\(m_1^i\,x_1\)}\\
B \,e^{\i\,\(m_1^t\,x_1+c_0\,x_1\)}
\end{matrix}
\)
=0.
\]
Since $c\neq 0$,  the vector $\(
\begin{matrix}
A\,e^{\i\,\(m_1^i\,x_1\)}\\
B \,e^{\i\,\(m_1^t\,x_1+c_0\,x_1\)}
\end{matrix}
\)\neq 0$, then we obtain $c_0=m_1^i-m_1^t$. Since we assumed $c_0=m_1^i-m_1^r$, we get that $m_1^r=m_1^t$ contradicting the assumption that $m_1^r\neq m_1^t$.
If on the other hand, $m_1^i-m_1^t=c_0$ proceeding in the same way and now using that $b\neq 0$ we get as before $m_1^r=m_1^t$ contradicting the initial assumption.
 
To prove that $m_2^r=m_2^t$, we proceed in the same way but using \eqref{eq: equation equals zero for second components BIS}.

To complete the proof of (i), since $m_1^t= m_1^r$ and $m_2^t=m_2^r$, substituting these into \eqref{eq:general exponential equation a b c} yields
\begin{equation}\label{eq:a exponential + (b+c) exponential =0}
a \,e^{\i\,\(m_1^i  x_1+m_2^i x_2\)}
+(b+c) \,e^{\i\,\(m_1^r x_1+m_2^r x_2+\psi(x_1,x_2)\)}=0.
\end{equation}
Differentiating this identity with respect to $x_1$ \cristian{and dividing by $\i$} yields
\begin{equation*}
a \,m_1^i\,e^{\i\,\(m_1^i  x_1+m_2^i x_2\)}
+(b+c) \,\(m_1^r+\psi_{x_1}\)\,e^{\i\,\(m_1^r x_1+m_2^r x_2+\psi(x_1,x_2)\)}=0.
\end{equation*}
That is, we obtain the system of equations
\[
\(
\begin{matrix}
1 & 1\\
m_1^i & m_1^r+\psi_{x_1}
\end{matrix}
\)
\(
\begin{matrix}
a \,e^{\i\,\(m_1^i  x_1+m_2^i x_2\)}\\
(b+c)\,e^{\i\,\(m_1^r x_1+m_2^r x_2+\psi(x_1,x_2)\)}
\end{matrix}
\)=0.
\]
Since $a\neq 0$, the vector $\(\begin{matrix}
a \,e^{\i\,\(m_1^i  x_1+m_2^i x_2\)}\\
(b+c)\,e^{\i\,\(m_1^r x_1+m_2^r x_2+\psi(x_1,x_2)\)}
\end{matrix}\)$ is a non trivial solution to the system 
and so $\det \(
\begin{matrix}
1 & 1\\
m_1^i & m_1^r+\psi_{x_1}
\end{matrix}
\)=m_1^r+\psi_{x_1}-m_1^i=0$ for all $x_1,x_2$ as desired.

Differentiating \eqref{eq:a exponential + (b+c) exponential =0} with respect to $x_2$ in the same way we obtain that 
$m_2^r+\psi_{x_2}-m_2^i=0$ completing the proof of (i).

{\it Proof of (ii).} Since $a=0$, we have simplifying the exponential $e^{\i\,\psi(X)}$ in \eqref{eq:general exponential equation a b c} that 
\begin{equation*}
b \,e^{\i\,\m_r\cdot X}
+
c \,e^{\i\,\m_t\cdot X}=0.
\end{equation*}
Differentiating the last equation with respect to $x_1$ and letting $x_2=0$ yields
\[
b\,m_1^r \,e^{\i\,m_1^r\,x_1}
+
c \,m_1^t\,e^{\i\,m_1^t\,x_1}=0.
\]
So we get the system
\[
\(
\begin{matrix}
1 & 1\\
m_1^r & m_1^t
\end{matrix}
\)
\(
\begin{matrix}
b\,e^{\i\,m_1^r\,x_1}\\
c \,e^{\i\,m_1^t\,x_1}
\end{matrix}
\)
=0,
\]
and since $c\neq 0$ the determinant of the last matrix must be zero and so $m_1^r=m_1^t$.
Similarly, differentiating with respect to $x_2$ and letting $x_1=0$ we obtain that $m_2^r=m_2^t$.

{\it Proof of (iii).} We have
\begin{equation*}
a\,e^{\i\,\m_i\cdot X}
+c\,e^{\i\,\left(\m_t\cdot X+\psi(X)\right)}=0,
\end{equation*}
and differentiating with respect to $x_1$ \cristian{and dividing by $\i$} yields
\[
m_1^i\,a\,e^{\i\,\m_i\cdot X}
+\(m_1^t+\psi_{x_1}\)c\,e^{\i\,\left(\m_t\cdot X+\psi(X)\right)}=0,
\]
and so
\[
\(
\begin{matrix}
1 & 1\\
m_1^i & m_1^t+\psi_{x_1}
\end{matrix}
\)
\(
\begin{matrix}
a\,e^{\i\,\m_i\cdot X}\\
c \,e^{\i\,\(\m_t\cdot X+\psi(X)\)}
\end{matrix}
\)
=0.
\]
Since the vector $\(
\begin{matrix}
a\,e^{\i\,\m_i\cdot X}\\
c \,e^{\i\,\(\m_t\cdot X+\psi(X)\)}
\end{matrix}
\)\neq 0$, the determinant of the matrix equals zero, i.e., $m_1^t+\psi_{x_1}(x_1,x_2)=m_1^i$
for all $X$ and therefore $\psi_{x_1}$ is constant, that is, $\psi(x_1,x_2)=c_0\,x_1+g(x_2)$. 
Substituting the value of $\psi$ in the original equation yields
\begin{equation*}
a\,e^{\i\,\m_i\cdot X}
+c\,e^{\i\,\left(m_1^t\,x_1+m_2^t\,x_2+c_0\,x_1+g(x_2)\right)}=0,
\end{equation*}
which differentiated with respect to $x_2$ gives
\begin{equation*}
m_2^i\,a\,e^{\i\,\m_i\cdot X}
+\(m_2^t+g'(x_2)\)\,c\,e^{\i\,\left(m_1^t\,x_1+m_2^t\,x_2+c_0\,x_1+g(x_2)\right)}=0.
\end{equation*}
Therefore 
\[
\(
\begin{matrix}
1 & 1\\
m_2^i & m_2^t+g'(x_2)
\end{matrix}
\)
\(
\begin{matrix}
a\,e^{\i\,\m_i\cdot X}\\
c \,e^{\i\,\(m_1^t\,x_1+m_2^t\,x_2+c_0\,x_1+g(x_2)\)}
\end{matrix}
\)
=0.
\]
Again, since the vector $\(
\begin{matrix}
a\,e^{\i\,\m_i\cdot X}\\
c \,e^{\i\,\(m_1^t\,x_1+m_2^t\,x_2+c_0\,x_1+g(x_2)\)}
\end{matrix}
\)\neq 0$, we obtain $m_2^t+g'(x_2)=m_2^i$, implying that $g'(x_2)$ is constant and so $g(x_2)=c_1\,x_2+c_2$ and we are done.

This proof of (iv) is the same as the proof of (iii).
\end{proof}

\section{Conclusion}
An analysis of the Maxwell system of electrodynamics in the context of distributions is carried out, leading to the derivation of boundary conditions for the electromagnetic field when the current and charge densities are localized at the interface. Consequently, by representing the electric field as a nonlinear perturbation of a plane wave characterized by a phase discontinuity function, the generalized Snell law is obtained. Furthermore, we derive formulas for the amplitudes of the reflected and transmitted waves in terms of the amplitude of the incident wave.

%

\bibliography{references}
\bibliographystyle{amsalpha}
\nocite{2017gutierrez:invariance}
\nocite{2011idemen:book}
\end{document}